\documentclass{article}
\usepackage{graphicx}
\usepackage{amsmath, amsthm, amssymb, appendix, bm, hyperref, mathrsfs}
\usepackage{subcaption} 
\usepackage{cite}
\newtheorem{proposition}{Proposition}[section]
\newtheorem{discussioncase}{Discussion}
\usepackage{caption}
\title{Iterative Generation and Generalized Degree Distribution of Higher-Order Fractal Scale-Free Networks}

\author{%
LIN QI\textsuperscript{1}%
\and 
JIA-XIN ZHANG\textsuperscript{1,}\textsuperscript{*}%
\\
\small \textsuperscript{1}School of Management Science and Engineering, Beijing Information Science and Technology University, \\
\small Beijing 102206, China\\
\small \textsuperscript{*}Corresponding author: jiaxinz@bistu.edu.cn
}

\date{}

\begin{document}

\maketitle

\begin{abstract}
Fractals represent one of the fundamental manifestations of complexity, and fractal networks serve as tools for characterizing and investigating the fractal structures and properties of large-scale systems. Higher-order networks have emerged as a research hotspot due to their ability to express interactions among multiple nodes. This study proposes an iterative generation model for higher-order fractal networks. The iteration is controlled by three parameters: the dimension $K$ of the simplicial complex, the multiplier $m$, and the iteration count $t$. The constructed network is a pure simplicial complex. Theoretical analysis using the similarity dimension and experimental verification using the box-counting dimension demonstrate that the generated networks exhibit fractal characteristics. When the multiplier $m$ is large, the generalized degree distribution of the generated networks is characterized by its scale-free nature. 
\end{abstract}


\section{Introduction}
Fractal theory provides a powerful framework for addressing fragmented and complex phenomena in nature. Within complex networks, numerous networks possess fractal characteristics, including hyperlinks on the World Wide Web, protein interaction networks, and co-occurrence networks of keyword terms in internet comment texts \cite{song_origins_2006, qi_internet_2022}. Numerous researchers have proposed distinct theoretical models of fractal networks and investigated their structural characteristics \cite{ma_understanding_2024}, yielding significant findings that reveal potential evolutionary drivers underlying the formation of fractal networks. 

Growing recognition of rich higher-order connectivity patterns in real-world systems \cite{lambiotte_networks_2019, benson_higher-order_2016, sizemore_knowledge_2018, giusti_twos_2016} is driving increased attention toward higher-order networks, as they can characterize interactions beyond the pairwise level. Concurrently, researchers have observed that certain networks within real-world systems exhibit both higher-order organization and fractal characteristics. Examples include partial phase synchronisation within simplicial complexes surrounding the brain's central regions \cite{tadic_multiscale_2022}, and information propagation within social hypernetworks \cite{luo_fractal_2024}. Furthermore, some studies simultaneously employ hypergraph and fractal theories to address problems \cite{8100392, ji_astute_2022}.

Network modelling is one of the key issues in network science, as it helps people understand how network structure influences network behaviour and can be used to improve network performance. Presently, few studies have explored models for constructing fractal networks based on higher-order structures. Researchers have developed distinct network models by classifying them according to two dimensions: the two different methods for constructing higher-order networks, and whether the network is fractal or pseudofractal. Ma et al. constructed a 2k uniform (1,3) flower hypernetwork model with fractal features \cite{ma_fractal_2021}. Zheng et al. demonstrated unconventional higher-order topological phenomena in fractal lattices featuring Sierpinski acoustic metamaterials \cite{zheng_observation_2022}. Xie et al. constructed pseudo-fractal simplicial complexes and investigated their combinatorial properties \cite{xie_combinatorial_2023}.

However, existing research on higher-order fractal networks has focused on simplicial complex of specific dimensions, or networks with relatively fixed hyperedge sizes and constrained fractal dimensions. To address these challenges, this study aims to establish an iterative generative model for higher-order fractal scale-free networks. We define an iterative algorithm whose behavior is determined by a triple of parameters $(K, m, t)$, representing the dimension of the simplicial complex, the multiplier, and the iteration count, respectively. The generated network is a pure $K$-dimensional simplicial complex. Its fractal dimension is controllable, and when the parameter $m$ is large, the network's generalized degree distribution exhibits scale-free properties.

\section{Preliminaries and Model Setup}
This section presents the fundamental concepts of higher-order networks and the model construction. 

\subsection{Fundamental Concepts of Higher-order Networks}
\label{subsec:Higher-order_networks}
A higher-order network is one which is capable of characterizing interactions among multiple nodes. One type comprises higher-order networks formed by hypernodes and hyperedges, while another consists of simplicial complex formed by simplices and their faces \cite{bianconi_higher-order_2021}. A set $\alpha = \{v_0, v_1, v_2, \cdots, v_d\}$ comprising $d + 1$ interacting nodes is termed a $d$-dimensional simplex $\alpha$, denoted as a $d$-simplex $\alpha$. For example, a 0-simplex is a point, a 1-simplex is an edge, a 2-simplex is a triangle, a 3-simplex is a tetrahedron, and so on. A face of a $d$-dimensional simplex $\alpha$ is a simplex $\beta$ formed by a proper subset of the simplex. The simplicial complex $\mathcal{K}$ is the set of simplices closed under inclusion of faces in each simplex, i.e. if $\alpha \in \mathcal{K}$ and $\beta \subset \alpha$, then $\beta \in \mathcal{K}$. The dimension $d$ of a simplicial complex is the maximum dimension of its simplices. In particular, a simplicial complex consisting of a set of $d$-dimensional simplices and their faces is a pure $d$-dimensional simplicial complex.

The $n$-skeleton of a $d$-dimensional simplicial complex $\mathcal{K}$ consists of all simplices of dimension at most $n$. Specifically, a 0-skeleton is the vertex set, while a 1-skeleton is denoted as $\mathcal{G} = \{\mathcal{V}, \mathcal{E}\}$, where $\mathcal{V}$ and $\mathcal{E}$ represent the vertex set and edge set of $\mathcal{G}$ respectively, as illustrated in Figure \ref{fig:skeleton}. A simplicial complex whose simplices are all the cliques (i.e., complete subgraphs) of the graph $\mathcal{G} = \{\mathcal{V}, \mathcal{E}\}$ is called the clique complex of the graph $\mathcal{G}$, as illustrated in Figure \ref{fig:skeleton}. The 1-skeleton of the clique complex of graph $\mathcal{G}$ is the graph $\mathcal{G}$ itself.

\begin{figure}[htbp]
    \centering  
    \begin{subfigure}[b]{0.3\textwidth}
        \centering
        \includegraphics[width=\linewidth]{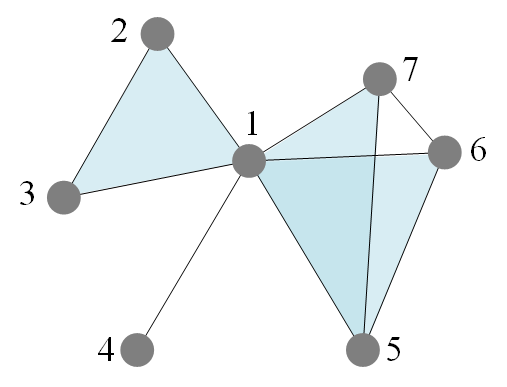}
        \caption{}
        \label{fig:skeleton-a}
    \end{subfigure}
    \begin{subfigure}[b]{0.3\textwidth}
        \centering
        \includegraphics[width=\linewidth]{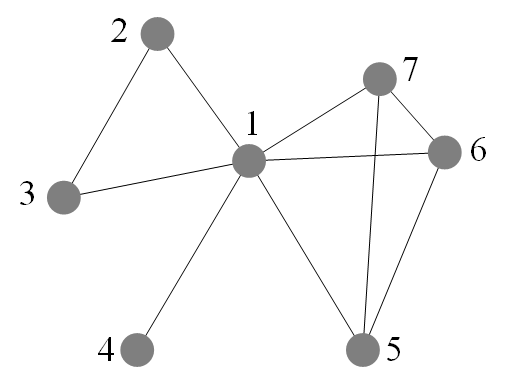}
        \caption{}
        \label{fig:skeleton-b}
    \end{subfigure}
    \begin{subfigure}[b]{0.3\textwidth}
        \centering
        \includegraphics[width=\linewidth]{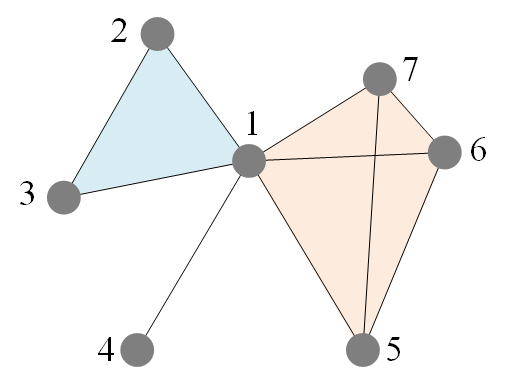}
        \caption{}
        \label{fig:skeleton-c}
    \end{subfigure}   
    \caption{\textbf{Simplicial Complex, 1-Skeleton, and Clique Complex. }(a) A two-dimensional simplicial complex, with blue triangles representing the 2-simplices it contains, edges denoting the 1-simplices it encompasses, and points signifying the 0-simplices it holds; (b) The 1-skeleton corresponding to the two-dimensional simplicial complex in (a) is also the 1-skeleton of the network in (c); (c) The clique complex generated by (b) is a simplicial complex composed of the 3-simplex $\{1,5,6,7\}$, the 2-simplex $\{1,2,3\}$, the 1-simplex $\{1,4\}$, and their faces. }
    \label{fig:skeleton}
\end{figure}

When studying network properties, one often focuses on the degree distribution, as it significantly influences the network's global dynamics. The concepts of degree and degree distribution can be naturally extended to simplicial complexes, where they are termed the generalized degree \cite{bianconi_network_2016} and the generalized degree distribution respectively. The generalized degree $k_{d,f}(\beta)$ denotes the number of $d$-simplices associated with an $f$-simplex $\beta$. When $d > f$, it denotes the number of $d$-simplices contained within the $f$-simplex $\beta$; when $d < f$, it denotes the number of $d$-dimensional faces within the $f$-simplex $\beta$, where $k_{d,f} = \binom{f + 1}{d + 1}$. In this paper, the probability $P_{d, l}(k)$ that any $l$-simplex in a pure $d$-dimensional simplicial complex $\mathcal{K}$ has a generalized degree of $k = k_{d,l}$ is defined as the $l$-dimensional generalized degree distribution of $\mathcal{K}$, where $l = 1, 2, \cdots, d - 1$. Subsequent sections shall investigate the generalized degree distributions of each dimension for the constructed network.

\subsection{Model}
The clique complex is a method for generating higher-order networks, treating all complete subgraphs within a graph as simplices within a simplicial complex to form a higher-order network. This study proposes a higher-order fractal network modelling approach that generates simplicial complexes through an iterative construction process. The specific network generation procedure is illustrated in Figure \ref{fig:T0T1}, as follows: 

\begin{enumerate}
    \item Initial: A $K$-dimensional simplex;
    \item At each time step: 
    \begin{enumerate}
        \item Identify all $K$-dimensional simplices within the current network (each complete subgraph comprising $K + 1$ nodes constitutes a $K$-dimensional simplex); 
        \item For each $K$-simplex $\alpha$, perform the following operation: 
        \begin{enumerate}
            \item Insert a new node (midpoint) within each edge of $\alpha$; 
            \item Starting from each vertex $v$ of $\alpha$, connect each of the $K$ midpoints adjacent to $v$ in pairs to form a $(K - 1)$-simplex $\beta$. For ease of notation, such a $(K - 1)$-simplex composed of midpoints is referred to herein as a bottom; 
            \item Introduce one new node $v'$, connected to each node of $\beta$ (repeat $m$ times. The parameter $m$ is termed the multiplier, where $m$ is a natural number, and the newly introduced node is referred to as the multiplication node); 
        \end{enumerate}
    \end{enumerate}
    \item When the number of iterations is less than $t$, repeat step 2. 
\end{enumerate}

When the dimension is $K$, the multiplier is $m$, and the iteration count is $t$, employing the method of clique complex treats all complete subgraphs in the graph as simplices, thereby yielding the simplicial complex $\mathcal{K}_{t}(K, m)$. This represents the union of all simplices within the network. The simplicial complex $\mathcal{K}_{t}(K, m)$ constitutes a pure $K$-dimensional simplicial complex, as it is composed of a set of $K$-simplices and their faces. Let the graph $\mathcal{G}_{t}(K, m)$ denote the 1-skeleton of $\mathcal{K}_{t}(K, m)$. 

One advantage of this model is that the dimension $K$ may be freely selected, thereby enabling the generation of simplicial complexes of various dimensions, as illustrated in Figure \ref{fig:tututu}.

\begin{figure}[htbp]
    \centering  
    \begin{subfigure}[b]{0.25\textwidth}
        \centering
        \includegraphics[width=\linewidth]{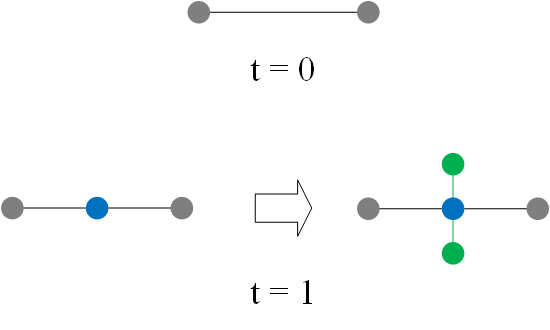}
        \caption{}
        \label{fig:T0T1-a}
    \end{subfigure}
    \qquad
    \begin{subfigure}[b]{0.3\textwidth}
        \centering
        \includegraphics[width=\linewidth]{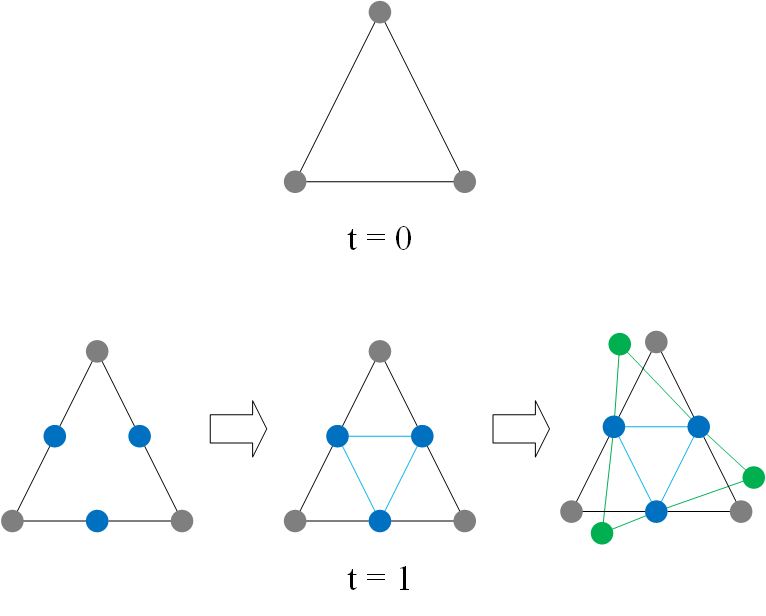}
        \caption{}
        \label{fig:T0T1-b}
    \end{subfigure}
    \begin{subfigure}[b]{0.3\textwidth}
        \centering
        \includegraphics[width=\linewidth]{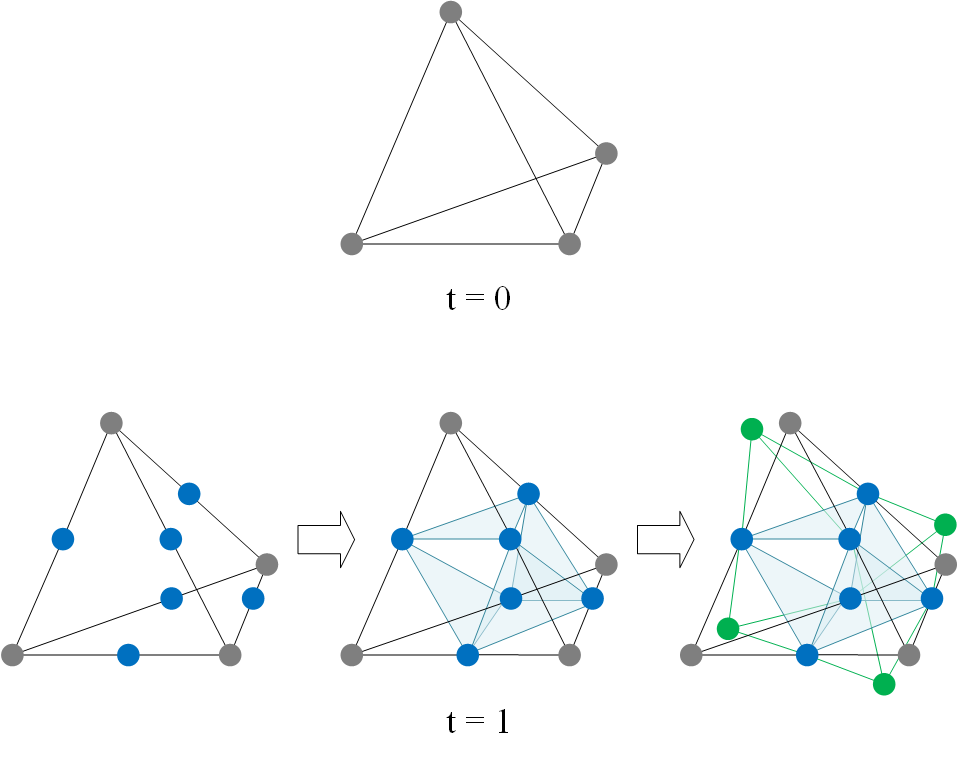}
        \caption{}
        \label{fig:T0T1-c}
    \end{subfigure}   
    \caption{\textbf{Schematic Diagram of the Iterative Construction Process. }The grey nodes in the diagram represent nodes present at the outset, the blue nodes denote midpoints inserted along edges during iteration, and the green nodes signify multiplier nodes. (a) Schematic of one iteration when $K = 1$ and $m = 2$, where the blue node form bottom; (b) Schematic of one iteration when $K = 2$ and $m = 1$, and the blue edges form bottoms; (c)Schematic of one iteration when $K = 3$ and $m = 1$, and the blue triangles form bottoms. }
    \label{fig:T0T1}
\end{figure}

\begin{figure}[htbp]
    \centering  
    \begin{subfigure}[b]{0.3\textwidth}
        \centering
        \includegraphics[width=\linewidth]{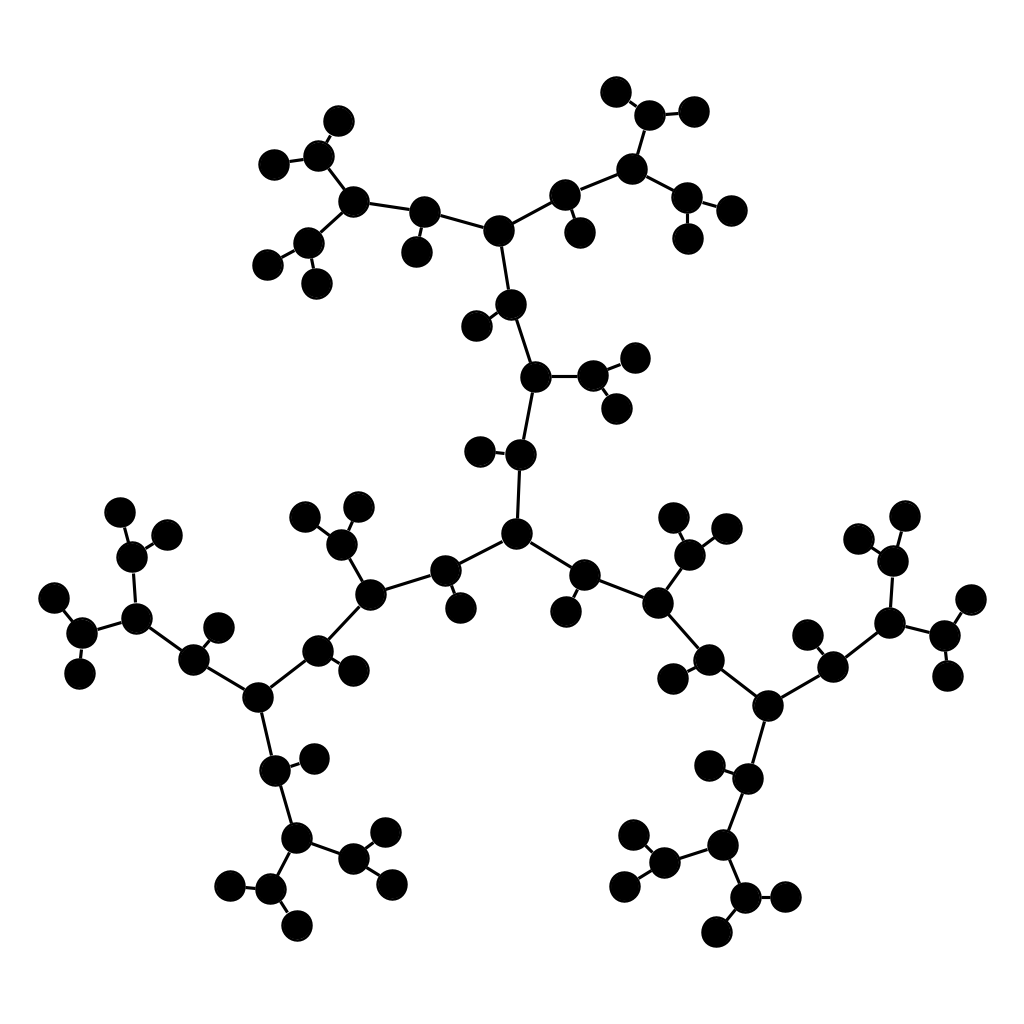}
        \caption{$\mathcal{G}_4(1, 1)$}
    \end{subfigure}
    \begin{subfigure}[b]{0.3\textwidth}
        \centering
        \includegraphics[width=\linewidth]{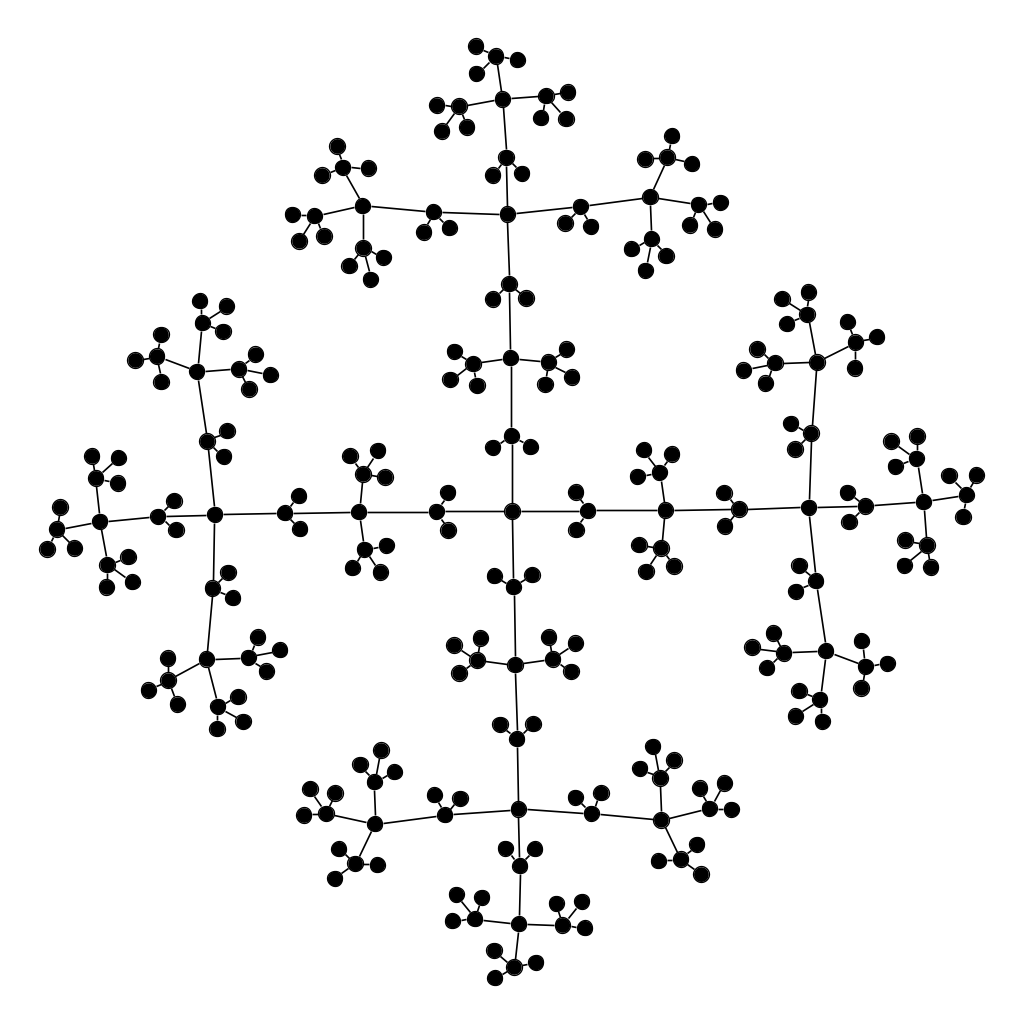}
        \caption{$\mathcal{G}_4(1, 2)$}
    \end{subfigure}
    \begin{subfigure}[b]{0.3\textwidth}
        \centering
        \includegraphics[width=\linewidth]{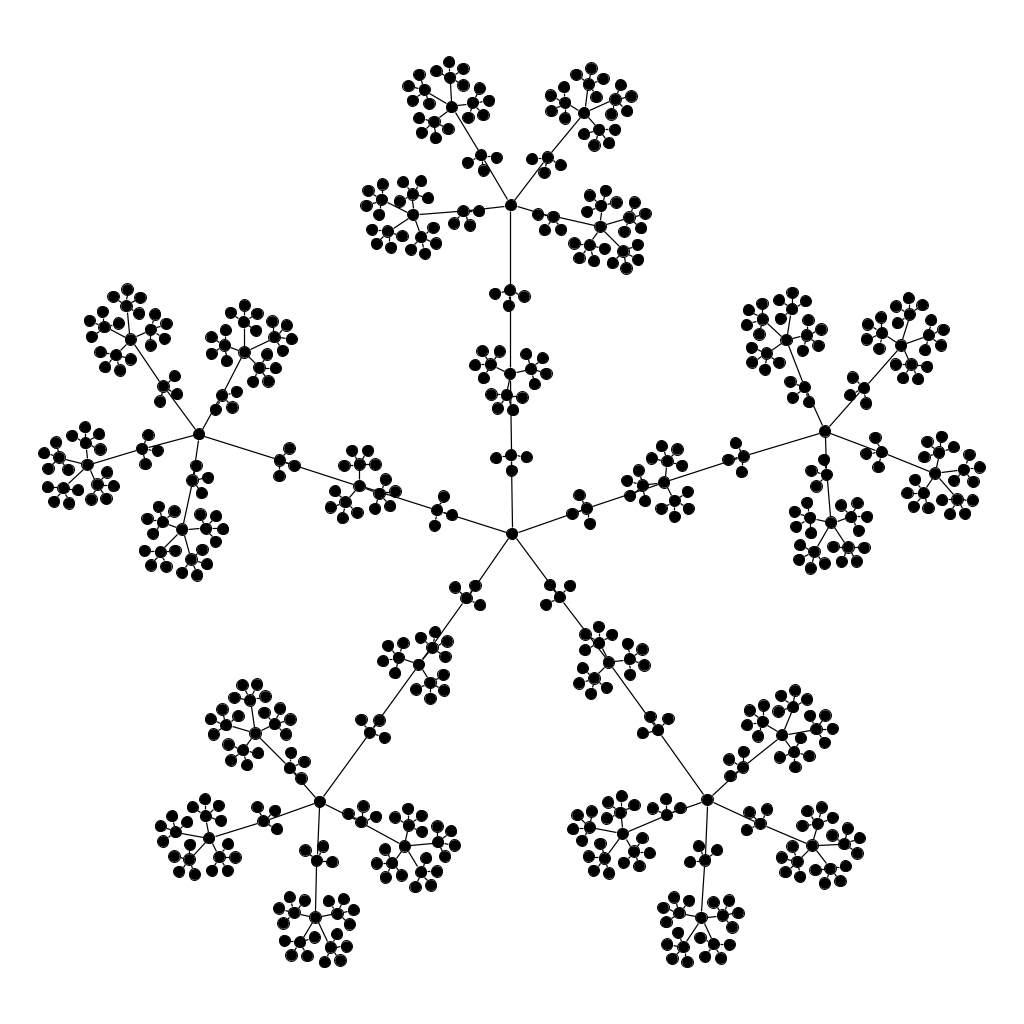}
        \caption{$\mathcal{G}_4(1, 3)$}
    \end{subfigure}   
    \begin{subfigure}[b]{0.3\textwidth}
        \centering
        \includegraphics[width=\linewidth]{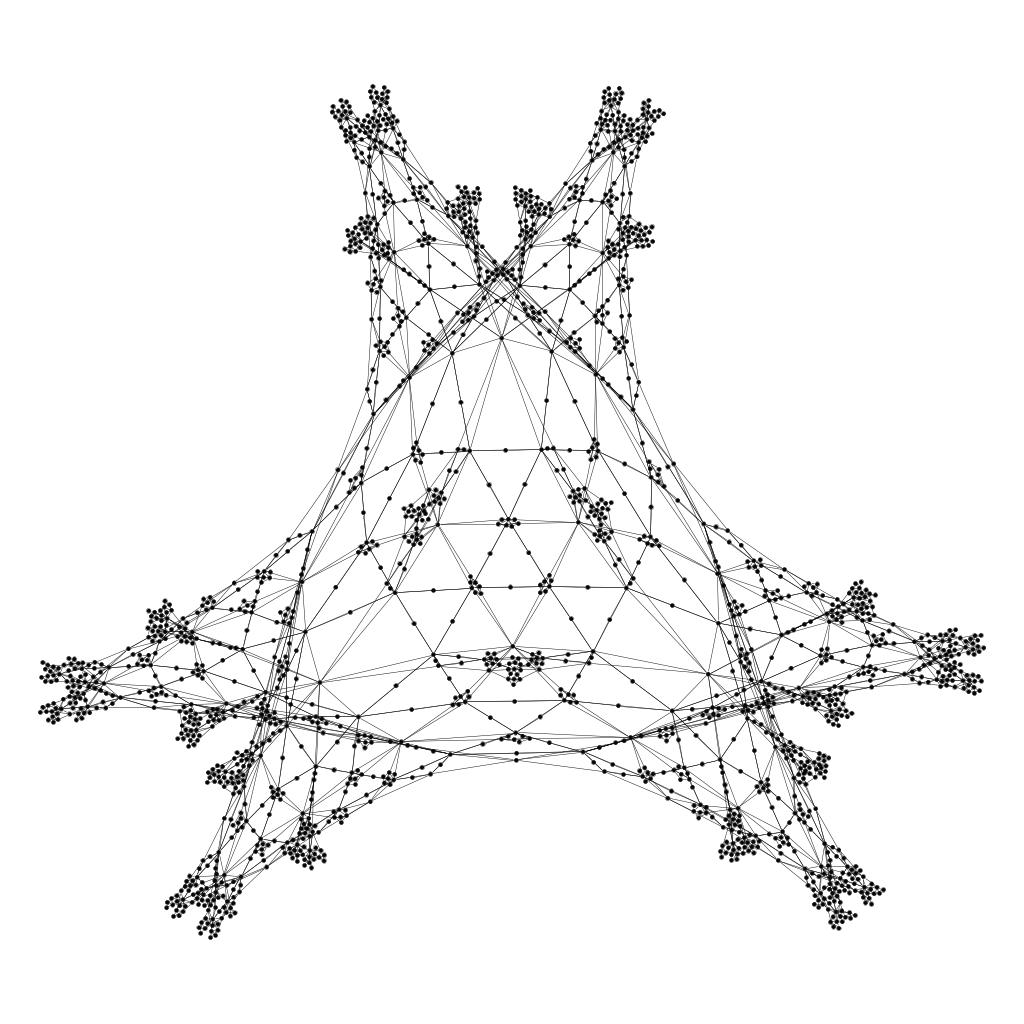}
        \caption{$\mathcal{G}_4(2, 1)$}
    \end{subfigure}
    \begin{subfigure}[b]{0.3\textwidth}
        \centering
        \includegraphics[width=\linewidth]{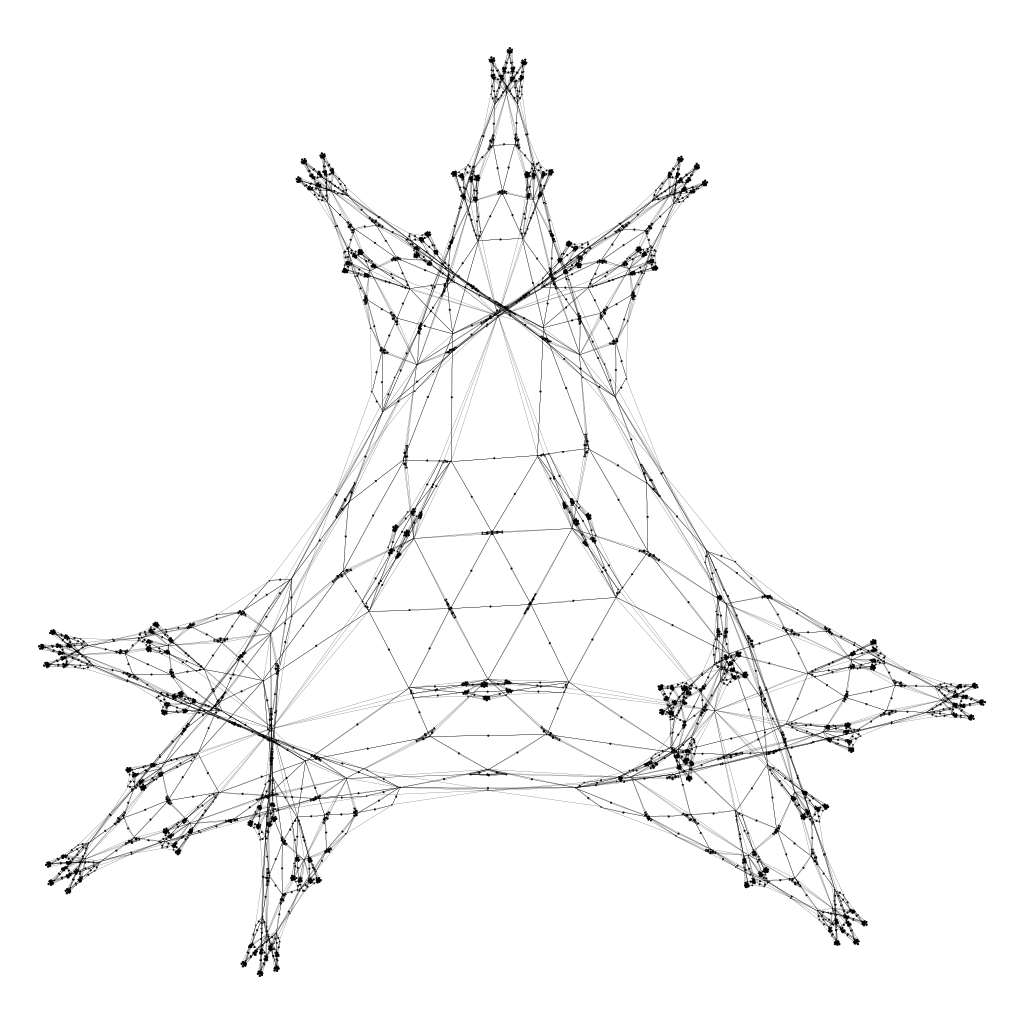}
        \caption{$\mathcal{G}_4(2, 2)$}
    \end{subfigure}
    \begin{subfigure}[b]{0.3\textwidth}
        \centering
        \includegraphics[width=\linewidth]{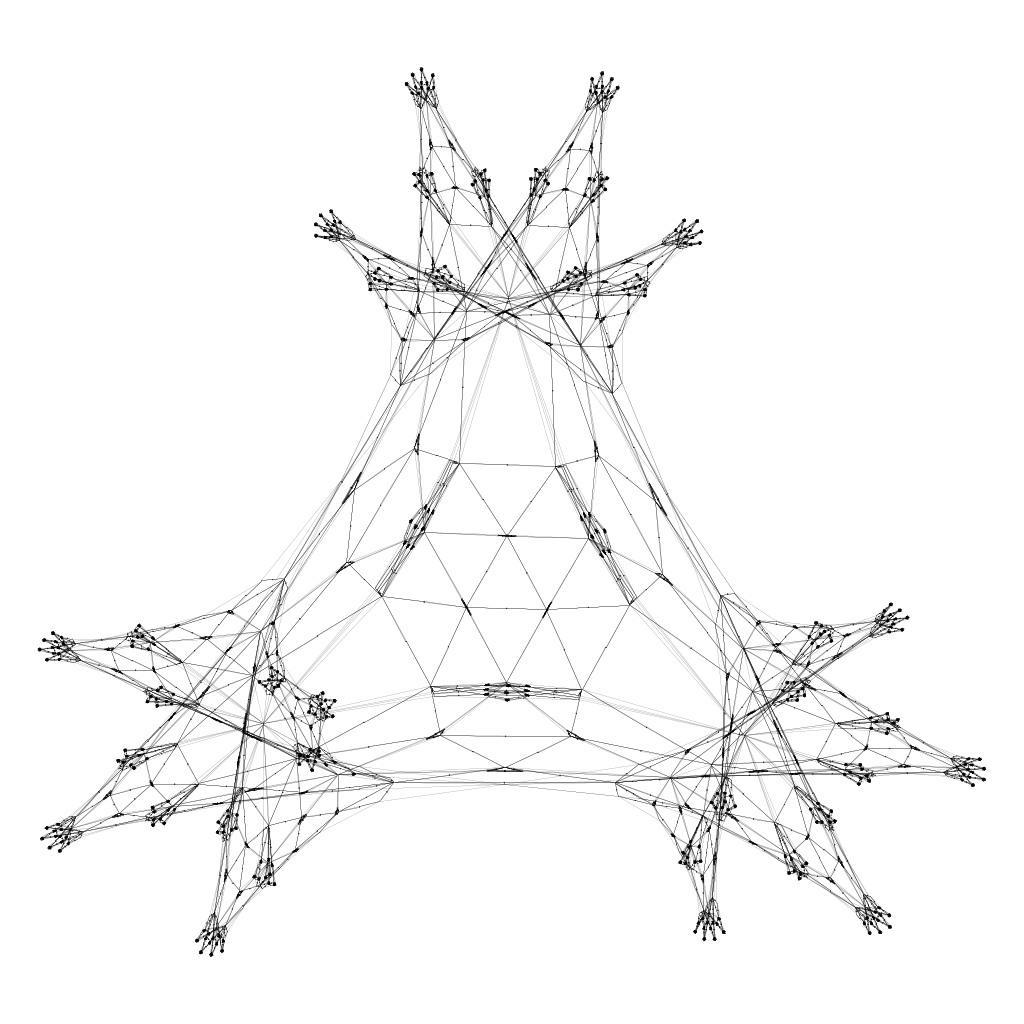}
        \caption{$\mathcal{G}_4(2, 3)$}
    \end{subfigure}
    \begin{subfigure}[b]{0.3\textwidth}
        \centering
        \includegraphics[width=\linewidth]{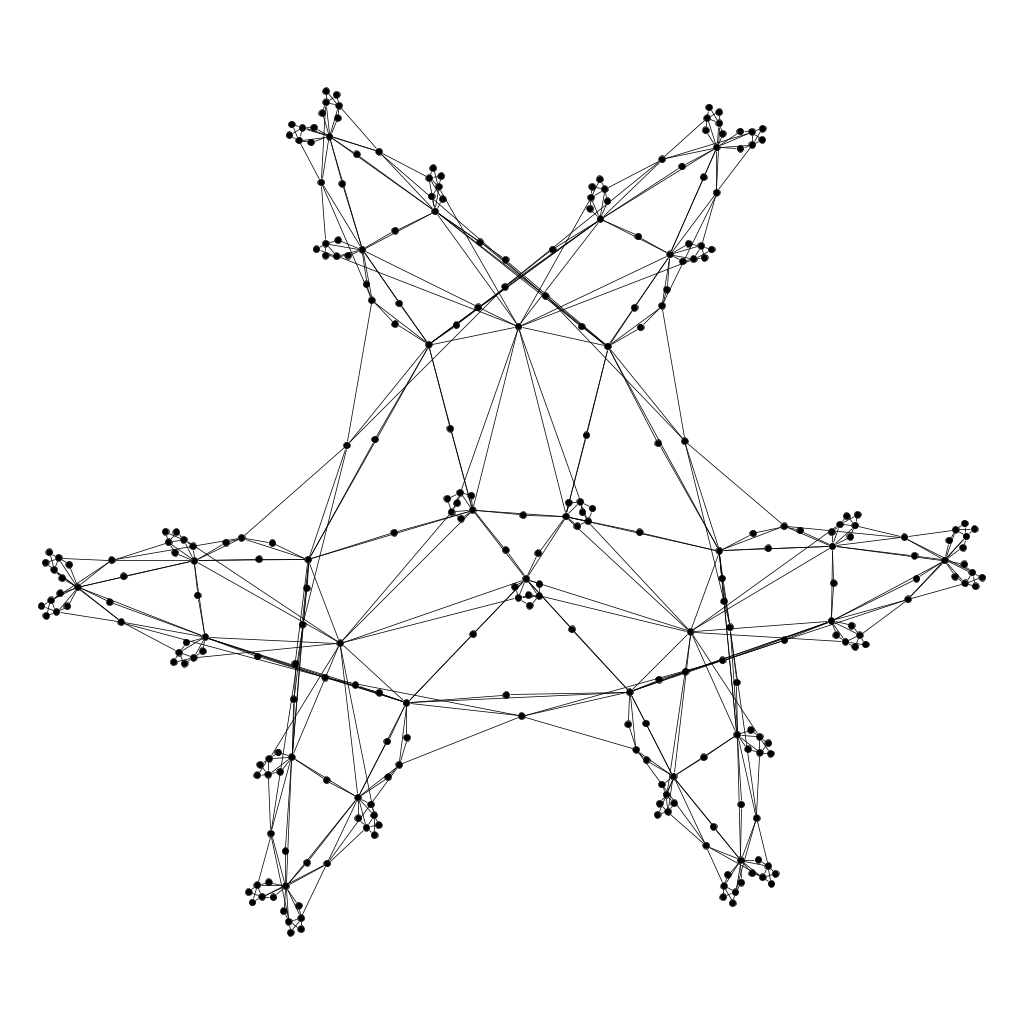}
        \caption{$\mathcal{G}_3(2, 1)$}
    \end{subfigure}
    \begin{subfigure}[b]{0.3\textwidth}
        \centering
        \includegraphics[width=\linewidth]{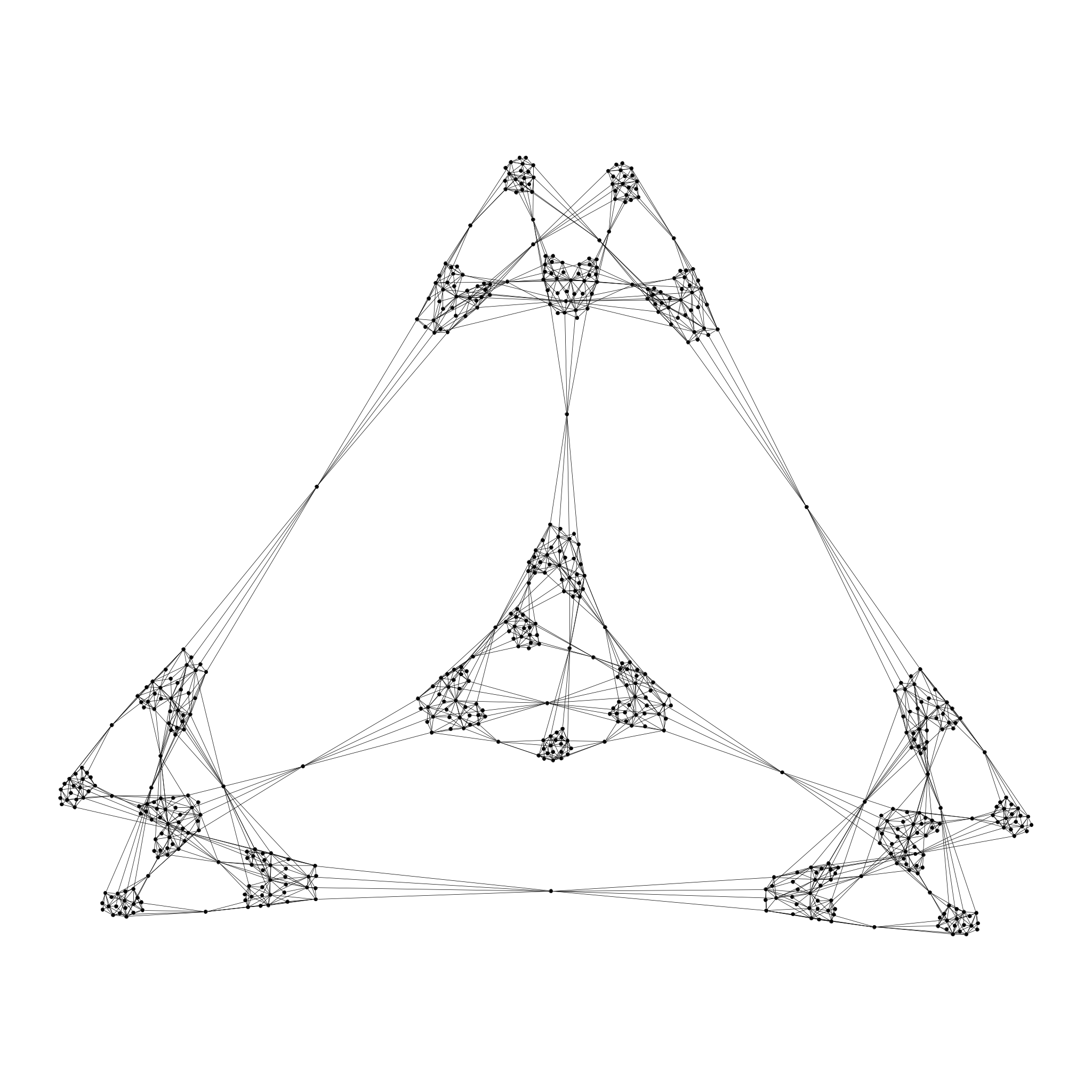}
        \caption{$\mathcal{G}_3(3, 1)$}
    \end{subfigure}
    \begin{subfigure}[b]{0.3\textwidth}
        \centering
        \includegraphics[width=\linewidth]{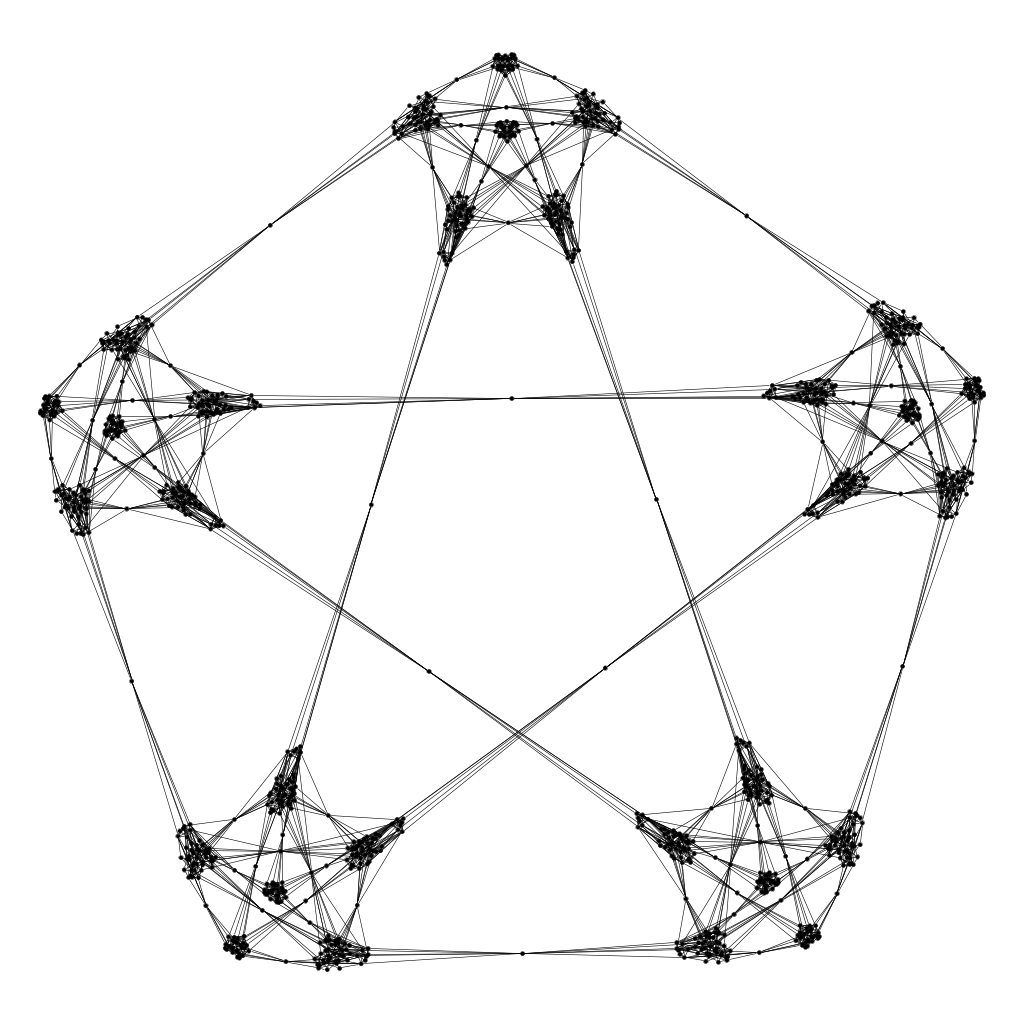}
        \caption{$\mathcal{G}_3(4, 1)$}
    \end{subfigure}
    \caption{\textbf{Networks Obtained After Multiple Iterations. }(a)-(c) Networks after four iterations with $K = 1$ for $m = 1, 2, 3$, respectively; (d)-(f) Networks after four iterations with $K = 2$ for $m = 1, 2, 3$, respectively; (g)-(i) Networks after three iterations with $m = 1$ for $K = 2, 3, 4$, respectively. }
    \label{fig:tututu}
\end{figure}

\section{Properties}
This section investigates certain properties of the higher-order network $\mathcal{K}_{t}(K, m)$ to elucidate its essential characteristics. Specifically, attention is focused on three properties: the number of $K$-simplices, the fractal dimension, and the generalized degree distribution. This is primarily because: (1) The number of $K$-simplices represents the primary metric for characterizing the scale of the network $\mathcal{K}_{t}(K, m)$; (2) By virtue of the iterative process, the network $\mathcal{K}_{t}(K, m)$ is inherently self-similar. We shall demonstrate below that it is also fractal, and the fractal dimension describes the dynamic evolution of fractal networks, hence necessitating its computation; (3) The generalized degree distributions of higher-order networks can be studied to determine whether they exhibit scale-free characteristics \cite{doi:10.1126/science.286.5439.509}. 

Now, let us begin by calculating the number of $K$-simplices contained within the pure $K$-dimensional simplicial complex $\mathcal{K}_{t}(K, m)$. 

\subsection{The Number of $K$-Dimensional Simplices}
According to the iterative generation rules of the network, the number of $K$-dimensional simplices exhibits exponential growth. Let $S$ denote the number of $K$-dimensional simplices contained in $\mathcal{K}_{t}(K, m)$ at time $t = 1$. According to the iterative process of the network, the number of $K$-simplices in $\mathcal{K}_{t+1}(K, m)$ at time $t+1$ is $S$ times that at time $t$. This is because each $K$-simplex at time $t$ will "transform" into a simplicial complex composed of $S$ $K$-simplices at the next time step. Therefore, the number of $K$-simplices in $\mathcal{K}_{t}(K, m)$ is $S^t$. $S$ is a function of dimension $K$ and multiplier $m$, and holds that: 
\begin{equation}
	S = 
	 \begin{cases} 
	m + 2,  & K = 1, \\
	3 m + 4, & K = 2, \\
	(m + 1) (K + 1), & K > 2. 
	 \end{cases}
	\label{eq:SKm}
\end{equation}
This is because the $K$-simplices in the $\mathcal{K}_{1}(K, m)$ can be divided into three categories, their number being the sum of the following three types of $K$-simplices: (1) $K$-simplices formed by a vertice at time $t = 0$ and the midpoints inserted at time $t = 1$, numbering $K + 1$; (2) $K$-simplices formed between midpoints. If the dimension $K$ of the simplicial complex is 2, then the number of such $K$-dimensional simplices is 1; otherwise, it is 0; and (3) the $K$-simplices formed by the bottom and the multiplier nodes. If the dimension $K$ of the simplicial complex is 1, then the number is $m$; otherwise, the number is $m \cdot (K + 1)$. 

\subsection{Fractal Dimension}
It is well known that for fractal figures, the fractal dimension is a crucial parameter used to describe the dynamic evolution of fractal networks. Numerous methods exist for calculating fractal dimensions within fractal geometry, one classic approach being the similarity dimension $d_s$ , frequently employed for regular self-similar geometric figures. If a whole set $\mathcal{A}$ can be partitioned into $N$ subsets of equal size, each of which is similar to the whole set $\mathcal{A}$ by a similarity ratio $r$, then the similarity dimension $d_s$ of $\mathcal{A}$ is defined as $d_s = \lim\limits_{r \to 0} (\log N / \log (1/r)) = - \lim\limits_{r \to 0} (\log N / \log r)$. In network science, there exist definitions for fractal networks and box-counting dimension. For a given box size $l_B$, the minimum number of boxes required to cover the network $\mathcal{G}$ is denoted by $N_B$. If a finite number $d_B$ exists such that $N_B  \varpropto \ell_B^{-d_B}$, then the network $\mathcal{G}$ is termed a fractal network, and $d_B$ is referred to as the box-counting dimension. 

\subsubsection{Similarity Dimension}
Clearly, when the network $\mathcal{K}_{t}(K, m)$ is partitioned into $S^n$ subsets, the similarity ratio between each subnetwork and the entire network is $1/2^n$. Therefore, the similarity dimension $d_s$  of $\mathcal{K}_{t}(K, m)$ is
\begin{align}
	d_s &= - \lim\limits_{n \to \infty} \frac{\log S^n}{\log 2^n} \notag \\ 
	&= \frac{\log S}{\log 2} 
	\label{eq:ds}
\end{align}
where $S$ is given by Equation \eqref{eq:SKm}. $d_s = \log (m+2) / \log 2$ for $K = 1$, $d_s = \log (3m+4) / \log 2$ for $K = 2$ and $d_s = \log (m+1)(K+1) / \log 2$ for $K \ge 3$. The theoretical value of the fractal dimension of the network, namely its similarity dimension, varies with two given parameters: the dimension $K$ of the simplicial complex and the multiplier $m$. Consequently, the fractal dimension of the higher-order fractal networks proposed in this study is controllable. 

\subsubsection{Box-counting Dimension}
The determination of the minimum number of boxes required given their sizes for the box-counting dimension is an NP-hard problem \cite{gallos_ReviewFractality_2007}. Various box-covering algorithms have been proposed to calculate the number of boxes needed for a covering network \cite{song_how_2007, sun_overlapping-box-covering_2014, guo_hub-collision_2022}. Based on whether the central node of the box is selected in advance, typical algorithms for computing box-counting dimensions are categorised into two types \cite{wen_fractal_2021}, and compact-box-burning (CBB) \cite{song_how_2007} and the overlapping-box-covering algorithm (OBCA) \cite{sun_overlapping-box-covering_2014} are representative algorithms for selecting central nodes and not selecting central nodes respectively. Employing the aforementioned algorithms CBB and OBCA, under various combinations of the two parameters $K$ and $m$, the fractal characteristics of the network were verified by averaging 100 independent replicate experiments \cite{zakar-polyak_towards_2023}. The resulting $d_{B}^{CBB}$ and $d_{B}^{OBCA}$ values are depicted in Figure \ref{fig:dsdB}. The computed box-counting dimension values are comparable to the similarity dimension.

\begin{figure}[htbp]
    \begin{subfigure}[b]{0.49\textwidth}
        \centering
        \includegraphics[width=\linewidth]{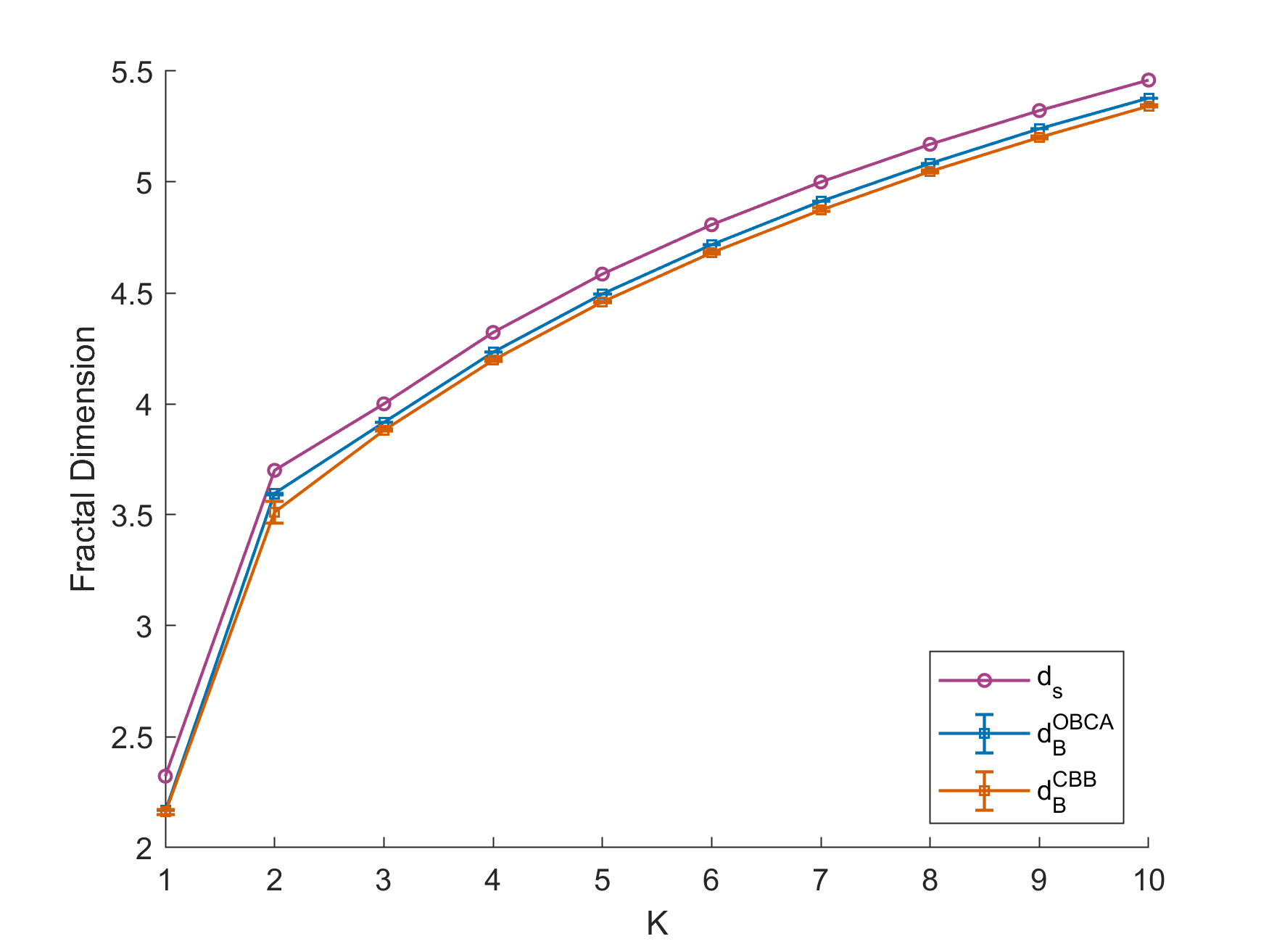}
        \caption{}
        \label{fig:dB_K}
    \end{subfigure}
    \begin{subfigure}[b]{0.49\textwidth}
        \centering
        \includegraphics[width=\linewidth]{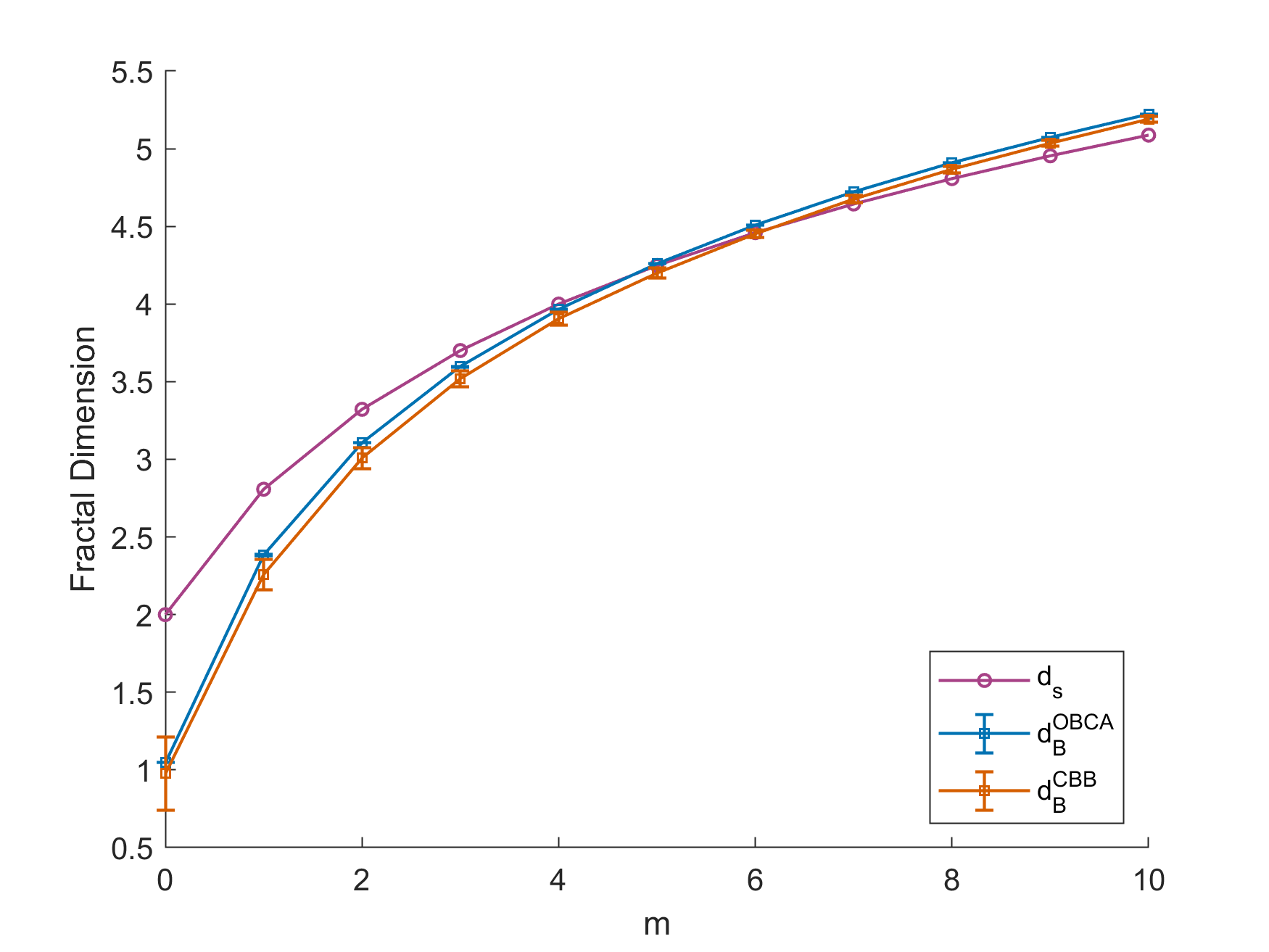}
        \caption{}
        \label{fig:dB_M}
    \end{subfigure}
    \caption{\textbf{Plot of the fractal dimension as a function of dimension $K$ and multiplier $m$. }(a) The fractal dimension of the graph generated with $m = 3$ after two iterations, for $K$ ranging from 1 to 10; (b) The fractal dimension of the graph generated with $K = 2$ after two iterations, for $m$ ranging from 0 to 10. }
    \label{fig:dsdB}
\end{figure}

\subsection{Generalized Degree Distribution}
In network science, degree distributions are extensively studied due to their impact on both the structure and functionality of entire networks\cite{10.1093/oso/9780198805090.001.0001}. Among various degree distributions, scale-free properties are prevalent across diverse networks, including the out-degree and in-degree distributions of the World Wide Web \cite{albert_diameter_1999}, actor networks, and power grids \cite{doi:10.1126/science.286.5439.509}. In higher-order networks, the generalized degree of a low-dimensional simplex represents its participation in higher-dimensional interactions, whilst the generalized degree distribution characterises the distribution of low-dimensional simplices engaging in such interactions. Next, we investigate the type of generalized degree distribution followed by the higher-order network $\mathcal{K}_{t}(K, m)$. Since $\mathcal{K}_{t}(K, m)$ is a pure simplicial complex, it encompasses simplices ranging from 0-dimensional to $K$-dimensional. Here, we investigate the generalized degree distribution for simplices of dimensions 1 to $K - 1$, where the definition of the generalized degree distribution is provided in subsection \ref{subsec:Higher-order_networks}. 

\textbf{Note: }The discussion pertains to the case where $K \ge 3$. The case where $K = 1$ is not addressed, as per the definition of the $l$-dimensional generalized degree distribution $P_{K, l}(k)$ for an $K$-dimensional simplicial complex in this paper, where $1 \le l \le K - 1$. Consequently, when the dimension $K = 1$, the discussion holds no significance. For the case where $K = 2$, its one-dimensional generalized degree distribution $P_{2, 1}(k)$ is in fact a two-point distribution, with the generalized degree taking only two values: 1 and $m + 2$. This discussion is not elaborated upon here. 

First, we investigate all possible values of the generalized degree for 1- to $(K - 1)$-dimensional simplices in the network and the corresponding number of simplices. 

\begin{proposition}
	When $t \ge K$, the generalized degree $k_{K,l}$ of an $l$-simplex in the network may take the value $(m + 1)^r$ ($r = 0, 1, 2, \cdots, K - L$), and the first occurrence time of an $l$-dimensional face whose generalized degree equals to $(m + 1)^r$ is $r$. Let $Y_{K}(l, t, r)$ denote the number of $l$-dimensional faces with generalized degree $k_{K,l} = (m + 1)^r$ at time $t$, and let $Y_{K}(l, t, 0)$ denote the number of $K$-simplexes. Then:

	 For $t = 0$,
	\begin{equation}
		Y_{K}(l, 0, r) =
		 \begin{cases} 
		\binom{K+1}{l+1},  & r=0, \\
		0, & \text{otherwise. }
		 \end{cases}
		 \label{eq:YKLt0}
	\end{equation}
	
	For $t \ge 0$ and $l = K$,
	\begin{equation}
		Y_{K}(K, t, r) =
		 \begin{cases} 
		S^t,  & r=0, \\
		0, & \text{otherwise. }
		 \end{cases}
		 \label{eq:YKKt}
	\end{equation}	

	For $t \ge 0$ and $l = 1, 2, 3, \dots, K - 1$,
	\begin{equation}
		Y_{K}(l, t, r) =
		 \begin{cases} 
		(l + 1) \cdot Y_{K}(l, t - 1, r) + (l + 2) \cdot Y_{K}(l + 1, t - 1, r - 1) ,  & 1 \le r \le K - L,\\
		(l + 1) \cdot Y_{K}(l, t - 1, r) + m \cdot (K+1) \cdot S^{t - 1} \cdot \binom{K}{l},  & r=0,\\
		0, & \text{otherwise. }
		 \end{cases}
		 \label{eq:YKLt_1}
	\end{equation}
	or, equivalently,
	\begin{equation}
		Y_{K}(l, t, r) =
			 \begin{cases} 
			(l + 1) \cdot Y_{K}(l, t - 1, r) + (l + 2) \cdot Y_{K}(l + 1, t - 1, r - 1) ,  & 1 \le r \le K - l,\\
			\binom{K + 1}{l + 1} \cdot S^t  - \sum_{j = 1}^{K - l} (m + 1)^j \cdot Y_{K}(l, t, j),  & r=0,\\
				0, & \text{otherwise. }
			 \end{cases}
		\label{eq:YKLt}
	\end{equation}
\end{proposition}

\begin{proof}
When $t = 0$, each $l$-dimensional face ($l = 1, 2, \cdots, K - 1$) in the network $\mathcal{K}_{t}(K, m)$ is contained within exactly one $K$-simplex, and their number is $\binom{K+1}{l+1}$, i.e., for $t = 0$, the number of $l$-simplices with generalized degree $k_{K,l} = 1$ is $\binom{K+1}{l+1}$, and the number of $l$-simplexes with other values of generalized degree $k_{K,l}$ is 0.

When $t \ge 1$, the $l + 1$ vertices of an $l$-simplex $\alpha$ in the graph ($l = 1, 2, \cdots, K - 1$) may be classified into the following three possibilities:

\begin{discussioncase}
$\alpha$ consists of a node at time $t - 1$ and $l$ nodes in a bottom at time $t$. 
\end{discussioncase}

At this point, $\alpha$ may be regarded as a new simplex generated from an $l$-simplex $\alpha'$ at time $t - 1$ (which has been partitioned by inserting midpoints along its edges). It consists of one vertex from $\alpha'$ and the $l$ adjacent midpoints. Consequently, the number of $K$-dimensional simplices contained within $\alpha$ remains identical to that of $\alpha'$, i.e., $k_{K,l}(\alpha) = k_{K,l}(\alpha')$. The resulting $l$-simplices thus constitutes $l + 1$ times the quantity of $l$-simplices with the same generalized degree as in the previous moment.

\begin{discussioncase}
$\alpha$ consists of $l + 1$ nodes in a bottom at time t. 
\end{discussioncase}

At this point, the vertices of $\alpha$ are the $l + 1$ midpoints adjacent to a single vertex of $\beta$, where $\beta$ is an $(l + 1)$-simplex at time $t - 1$. At time $t$, since $\alpha$ is contained within a bottom, $\alpha$ is connected to $m$ multiplier nodes and one node from time $t - 1$, hence $k_{K,l}(\alpha) = (m + 1) \cdot k_{K,l+1}(\beta)$. The number of $l$-simplices obtained is $l + 2$ times the number of $(l + 1)$-simplices from the previous moment;

\begin{discussioncase}
$\alpha$ consists of $l$ nodes in a bottom at time $t$ and a multiplied node connected to this bottom.
\end{discussioncase}

For $\alpha$, we have $k_{K,l}(\alpha) = 1$ because $\alpha$ is contained only in the $K$-simplex formed by this multiplier node and the bottom. The number of $l$-simplices obtained is $m \cdot b(K - 1) \cdot \binom{K}{l}$ times the number of $K$-simplices at this moment. Moreover, the number of $l$-simplices of this type may also be obtained by subtracting the product of the number of $l$-simplices of all other generalities and their respective generalities from the total number of $l$-simplices (counting with repetition). 

For a $(K - 1)$-dimensional simplex $\alpha$ at time $t = 1$, if $\alpha$ is a bottom, then its generalized degree $k_{K,K-1}(\alpha) = m + 1$; otherwise, $k_{K,K-1}(\alpha) = 1$.

By synthesising the above analysis, the preceding formula can be derived. 
\end{proof}

The degree distribution of many real-world networks is scale-free, meaning that the probability $P(k)$ of any randomly selected node having degree $k$ follows a power-law relationship: $P(k) \sim k^{-\gamma}$. When $m$ is sufficiently large, the network $\mathcal{K}_{t}(K, m)$ constructed in this paper becomes a scale-free network, whose degree distribution $P_{K, l}(k)$ follows a power-law distribution. 

According to the formula \eqref{eq:YKLt} describing the variation of the number of simplices over time, when $r = 0$, the expression contains a term involving $S^t$. Since $S > 1$, $S^t$ grows exponentially with time $t$. Therefore, it may be assumed that all $Y_K(l,t,r)$ will ultimately grow at a rate proportional to $S^t$ multiplied by some constant, i.e., assuming $Y_K(l,t,r) \approx c_{l,r} S^t$, where $c_{l,r}$ is a constant independent of time $t$ representing the growth factor. This analysis enables us to obtain the following proposition about the relations among the numbers of $l$-simplices at different generalized degrees. 

\begin{proposition}
	When $t$ is sufficiently large, the ratio of the number of $l$-simplices with generalized degree $(m + 1)^r$ to the number of $(l + 1)$-simplices with generalized degree $(m + 1)^{r - 1}$ becomes a constant value. That is, the relationship between $c_{l, r}$ and $c_{l+1, r-1}$ is given by:
	\begin{equation}
		c_{l, r} = \frac{l + 2}{S - l - 1} c_{l+1, r-1}
		\label{eq:cLLp1}
	\end{equation}

	In particular, $c_{l, 1} = \frac{l + 2}{S - l - 1} c_{l+1, 0}$ when $r = 1$. 
\end{proposition}

\begin{proof}
	Substituting $Y_K(l,t,r) \approx c_{l,r} S^t$ into the branch of Equation \eqref{eq:YKLt} for the case $1 \le r \le K - l$ yields: 
	\begin{equation}
		c_{l, r} S^t = (l + 1) c_{l, r} S^{t-1} + (l + 2) c_{l+1, r-1} S^{t-1}
	\end{equation}
	Rearranging the above expression yields Equation \eqref{eq:cLLp1}. 
\end{proof}

\begin{proposition}
	The growth rate of the number of l-simplices with generalized degree 1 is
	\begin{align}
	    c_{l, 0} = \binom{K + 1}{l + 1} \cdot \frac{m (l + 1)}{S - l - 1}
	      \label{eq:cL0}
	\end{align}
\end{proposition}

\begin{proof}
	Substituting Equation \eqref{eq:cLLp1} into the branch of Equation \eqref{eq:YKLt} for the case $r = 0$ yields: 
	\begin{equation}
		c_{l, 0} S^t = \binom{K + 1}{l + 1} S^t  - \sum_{j = 1}^{K - l} (m + 1)^j  c_{l, j} S^t
	\end{equation}
	Dividing both sides of the equation by $S^t$ yields
	\begin{equation}
		c_{l, 0} = \binom{K + 1}{l + 1} -  \sum_{j = 1}^{K - l} (m + 1)^j  c_{l, j}
		\label{eq:cL0mid}
	\end{equation}
	This is equivalent to
	\begin{equation}
		\sum_{j = 0}^{K - L} (m + 1)^j  c_{L, j} = \binom{K + 1}{L + 1}
		\label{eq:sumclj}
	\end{equation}
	Substituting \eqref{eq:cLLp1} into the right-hand side of the resulting equation \eqref{eq:cL0mid} and performing some straightforward arithmetic operations yields
	\begin{equation}
		c_{l, 0} = \binom{K + 1}{l + 1} - \frac{l + 2}{S - l - 1} \cdot (m + 1) \sum_{j = 0}^{K - l - 1}(m + 1)^j c_{l+1, j}
	\end{equation}
	Substituting the result obtained when \eqref{eq:sumclj} $L$ is taken as $l + 1$ and simplifying yields the value of $c_{l, 0}$ as shown in formula \eqref{eq:cL0}. 
\end{proof}

\begin{proposition}
	The transformation from low-dimensional to high-dimensional ratios of different generalized-degree faces is possible: 
	\begin{equation}
		\lim\limits_{t \to \infty} \frac{Y_K (l, t, r)}{Y_K (l, t, r + 1)} =  \lim\limits_{t \to \infty} \frac{Y_K (l+1, t, r-1)}{Y_K (l+1, t, r)}
	      \label{eq:YLtrYLp1trs1}
	\end{equation}
\end{proposition}

\begin{proof}
	By assumption $Y_K(l,t,r) \sim c_{l,r} S^t$: 
	\begin{equation}
		\lim\limits_{t \to \infty} \frac{Y_K (l, t, r)}{Y_K (l, t, r + 1)} = \frac{c_{l, r}}{c_{l, r+1}}
	\end{equation}
	and
	\begin{equation}
		\lim\limits_{t \to \infty} \frac{Y_K (l+1, t, r-1)}{Y_K (l+1, t, r)} = \frac{c_{l+1, r-1}}{c_{l+1, r}}
	\end{equation}
	Substituting Equation \eqref{eq:cLLp1} into the above, which completes the proof. 
\end{proof}

\begin{proposition}
	Let $C_l$ denote the limit of the ratio of the number of $l$-dimensional faces with generalized degree $1$ to that with generalized degree $m + 1$ as $t \to \infty$. Then
	\begin{equation}
		C_l = \frac{(S - l - 2) (l + 1)}{(K - l) (l + 2)}
	      \label{eq:CL}
	\end{equation}
\end{proposition}

\begin{proof}
	By the definition of $C_l$, we have $C_l = \lim\limits_{t \to \infty}\frac{Y_{K}(l, t, 0)}{Y_{K}(l, t, 1)}$. Substituting $Y_K(l,t,r) \sim c_{l,r} S^t$ yields 
	\begin{equation}
		C_l = \frac{c_{l, 0} S^t}{c_{l, 1} S^t}
	\end{equation}
	Substituting equation \eqref{eq:cLLp1} yields
	\begin{equation}
		C_l = \frac{S - l - 1}{l + 2} \cdot \frac{c_{l, 0}}{c_{l+1, 0}}
	\end{equation}
	From equation \eqref{eq:cL0}, the values of $c_{l, 0}$ and $c_{l + 1, 0}$ can be determined. Substituting these values and rearranging yields equation \eqref{eq:CL}. The value of $C_l$ has also been verified experimentally.
\end{proof}

\begin{proposition}
	As $t \to \infty$, the ratio of the number of $l$-dimensional faces with generalized degrees $(m + 1)^r$ and $(m + 1)^{r + 1}$ is given by $C_{l + r}$, namely
	\begin{equation}
		\lim\limits_{t \to \infty} \frac{Y_K (l, t, r)}{Y_K (l, t, r + 1)} =  C_{l + r}
		\label{eq:YLtrYLtrp1}
	\end{equation}
\end{proposition}

\begin{proof}
	Combining equations \eqref{eq:YLtrYLp1trs1} and \eqref{eq:CL} yields the result. This has also been verified experimentally.
\end{proof}

\begin{proposition}
	When $t$ is large and $m$ is large, the generalized degree distribution $P_{K, l}(k)$ of $l$-dimensional faces approximates a power-law distribution, namely
	\begin{align}
		P_{K, l}(k) \sim k^{-\gamma}
	      \label{eq:PKl}
	\end{align}
	Moreover, the approximate value of the exponent $\gamma$ in the power-law distribution is
	\begin{align}
		\gamma &\approx \frac{1}{(K - l) \log (m + 1)} \log \frac{l + 1}{K + 1} \binom{S - l - 2}{K - l}
	      \label{eq:gammaPKL}
	\end{align}
\end{proposition}

\begin{proof}
	When $m$ is large, $C_{l+r} = \left({(S-l-r-2)(l+r+1)}\right)/\left({(K-l-r)(l+r+2)}\right) \approx {S}/({K-l-r})$. Thus, 
	\begin{equation}
		Y_K(l, t, r)  \approx Y_K(l, t, 0) \cdot \prod\limits_{j=0}^{r-1} \frac{K - l - j}{S}
	\end{equation}
	and
	\begin{equation}
	\log Y_K(l, t, r) \approx \log Y_K(l, t, 0) + \sum\limits_{j = 0}^{r - 1} \log \frac{K - l -j}{S}
	\end{equation}

	Since $\log \frac{K - l -j}{S}$ varies little with $j$, the following ratio also changes little with $r$:
	\begin{align}
	\frac{\log P_{K, l}(k(r + 1)) - \log P_{K, l}(k(r))}{\log k(r + 1) - \log k(r)} = \frac{\log \frac{Y_K(l, t, r + 1)}{\sum_{j} Y_K(l, t, j)} - \log \frac{Y_K(l, t, r)}{\sum_{j} Y_K(l, t, j)}}{\log (m + 1)^{r + 1} - \log (m + 1)^r}
	\end{align}

	This indicates that in the double-logarithmic coordinate system, $(k, P_{K, l}(k))$ approximates a linear distribution, as illustrated in Figures \ref{fig:GDD1} and \ref{fig:GDD2}. Consequently, the generalized degree distribution $P_{K, l}(k)$ approximates a scale-free distribution, as expressed in equation \eqref{eq:PKl} within the proposition. 

	Based on the preceding analysis, the first and last points in the distribution may be treated as fitting targets to estimate the value of the power exponent $\gamma$. Combining Equations \eqref{eq:CL} and \eqref{eq:YLtrYLtrp1}, when $t$ is sufficiently large, the number of $l$-dimensional faces with generalized degree $k = k(0) = (m + 1)^0 = 1$ is given by
	\begin{equation}
		N_l (1) = Y_K(l, t, 0) \approx Y_K(l, t, K - L) \cdot \prod\limits_{j = l}^{K - 1} C_j = \frac{l + 1}{K + 1} \cdot \binom{S - l - 2}{K - l} \cdot Y_K(l, t, K - l)
		\label{eq:N1rmin}
	\end{equation}
	thereby
	\begin{equation}
		\gamma \approx -\frac{\log P_{K, l}((m + 1)^{K - l}) - \log P_{K, l}(1)}{\log (m + 1)^{K - l} - \log 1}
	\end{equation}
	Simplify to obtain $\gamma \approx \frac{1}{(K - l) \log (m + 1)} \log \frac{l + 1}{K + 1} \binom{S - l - 2}{K - l}$. Compared to the $\gamma$ value obtained by fitting all data points using the least squares method, the estimate derived from equation \eqref{eq:gammaPKL} shows no significant difference, as illustrated in Figures \ref{fig:GDD1} and \ref{fig:GDD2}. 
\end{proof}

\begin{figure}[htbp]
    \centering
    
    \begin{minipage}{0.04\textwidth}
        \centering
        \rotatebox{90}{Generalized degree distributions, $P_{K, 1}(k_{K, 1})$}
    \end{minipage}%
    \hfill
    \begin{minipage}{0.94\textwidth}
        \centering
        
        \begin{subfigure}{0.3\textwidth}
            \centering
            \includegraphics[width=\linewidth]{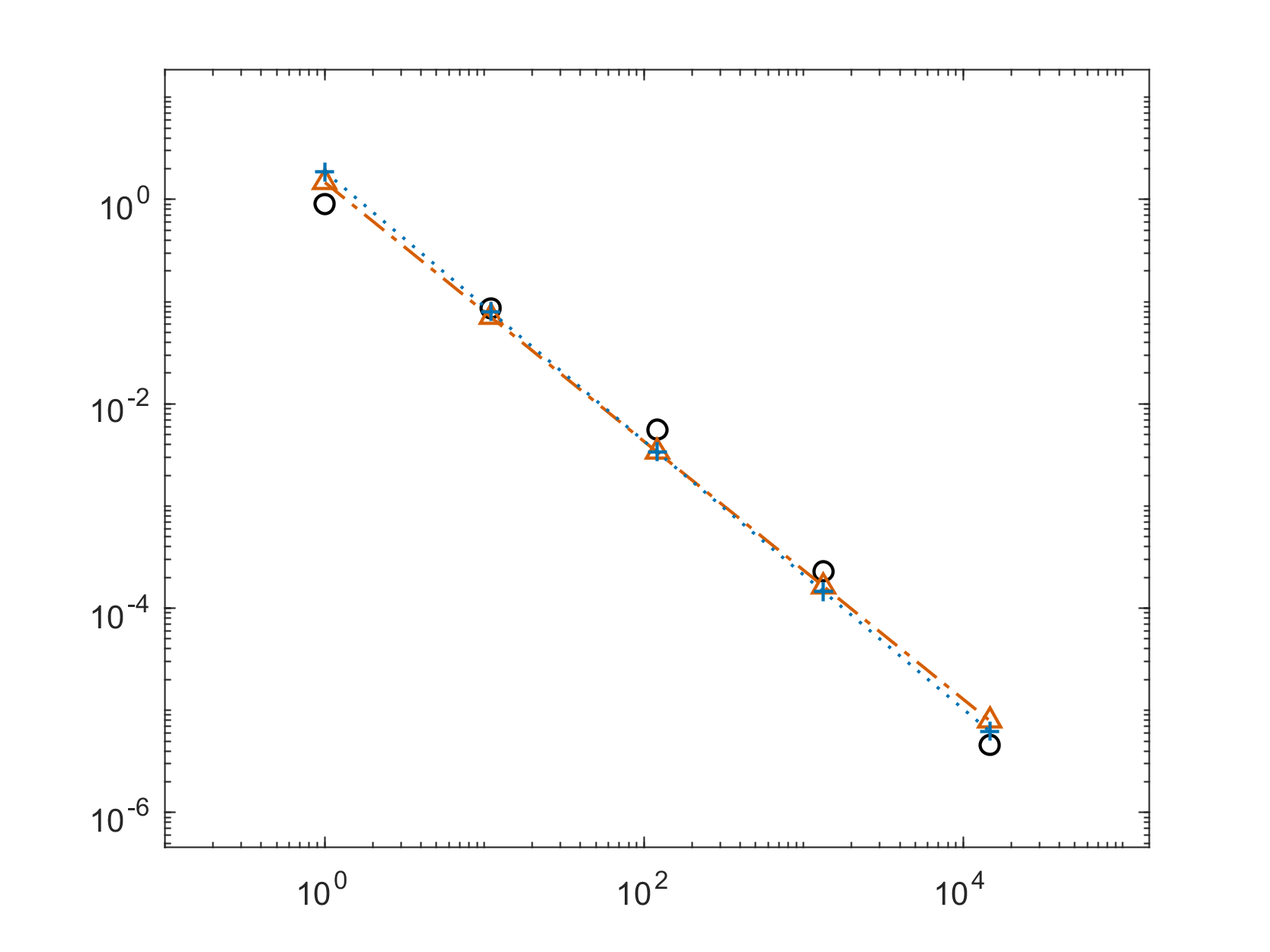}
            \caption{$K = 5,\,m = 10$}
        \end{subfigure}%
        \hfill
        \begin{subfigure}{0.3\textwidth}
            \centering
            \includegraphics[width=\linewidth]{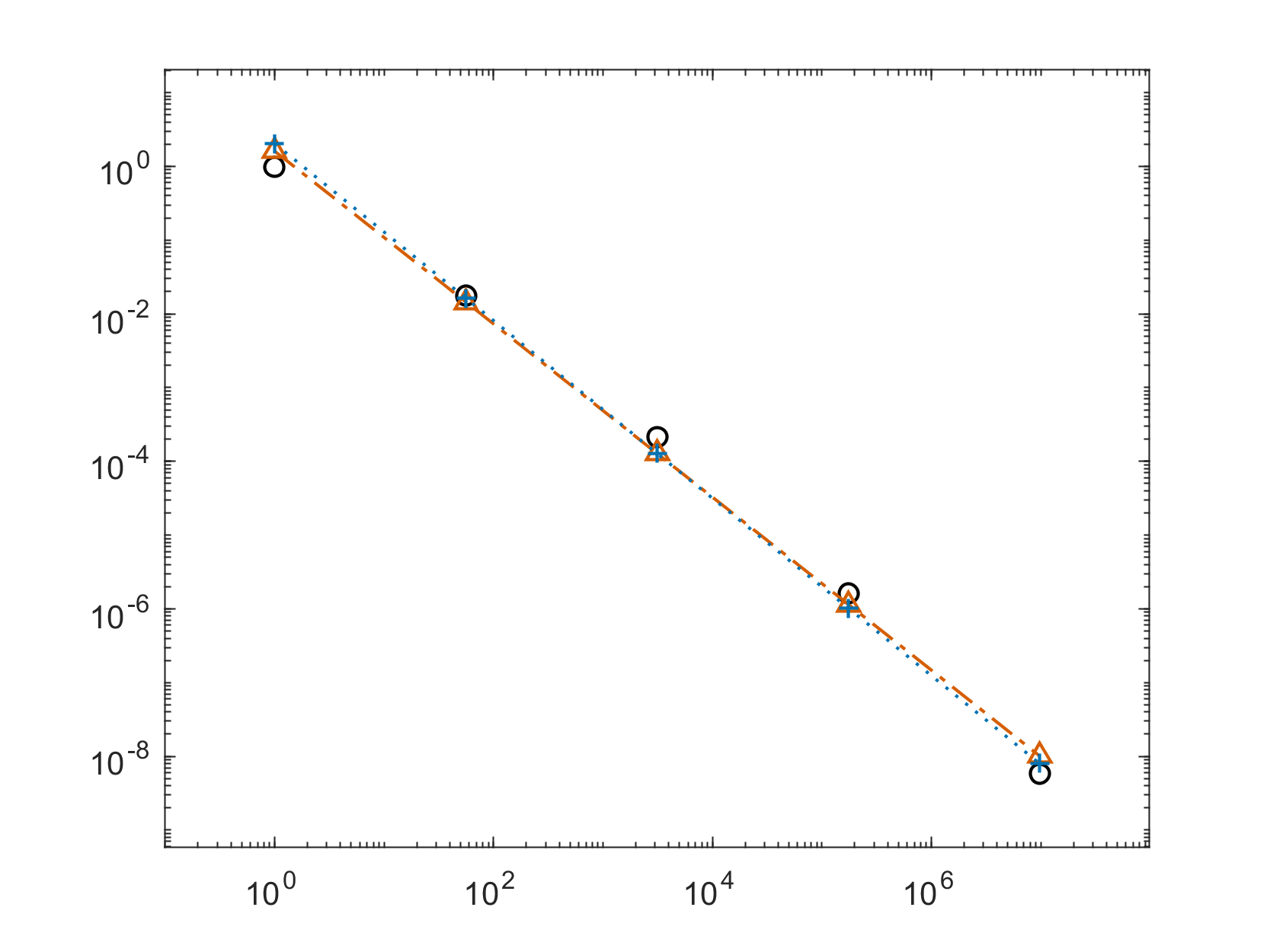}
            \caption{$K = 5,\,m = 55$}
        \end{subfigure}%
        \hfill
        \begin{subfigure}{0.3\textwidth}
            \centering
            \includegraphics[width=\linewidth]{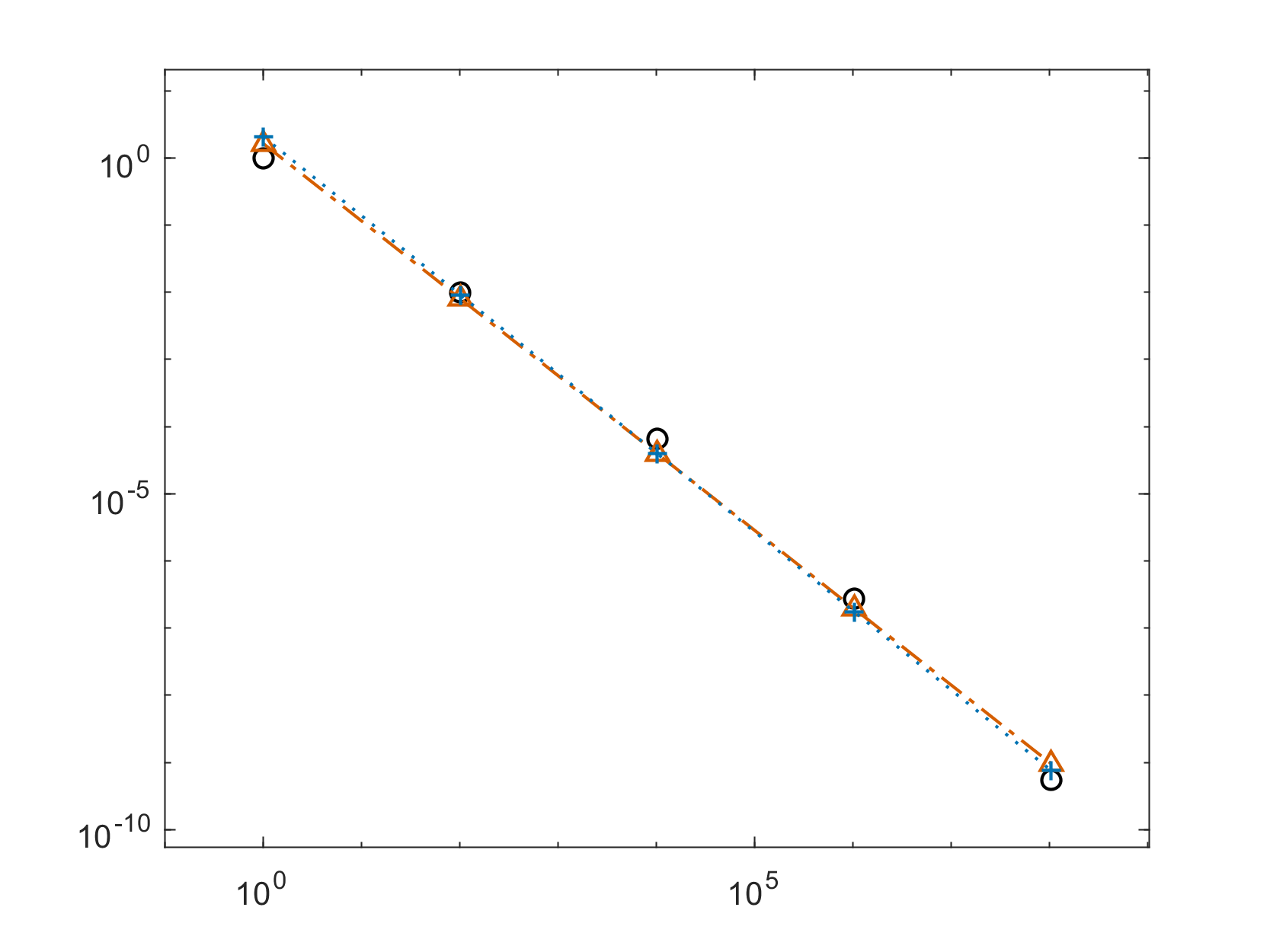}
            \caption{$K = 5,\,m = 100$}
        \end{subfigure}
        
        \begin{subfigure}{0.3\textwidth}
            \centering
            \includegraphics[width=\linewidth]{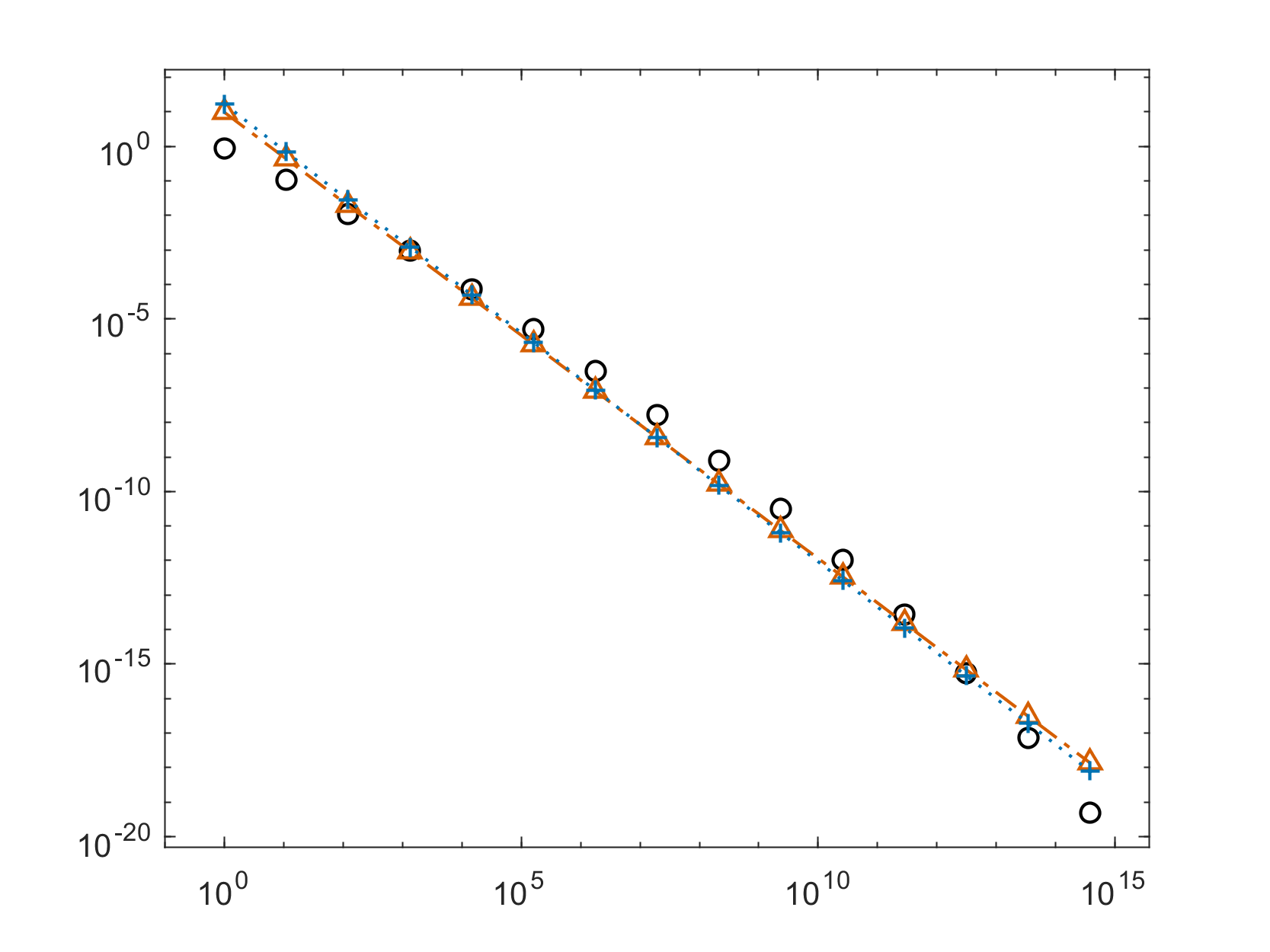}
            \caption{$K = 15,\,m = 10$}
        \end{subfigure}%
        \hfill
        \begin{subfigure}{0.3\textwidth}
            \centering
            \includegraphics[width=\linewidth]{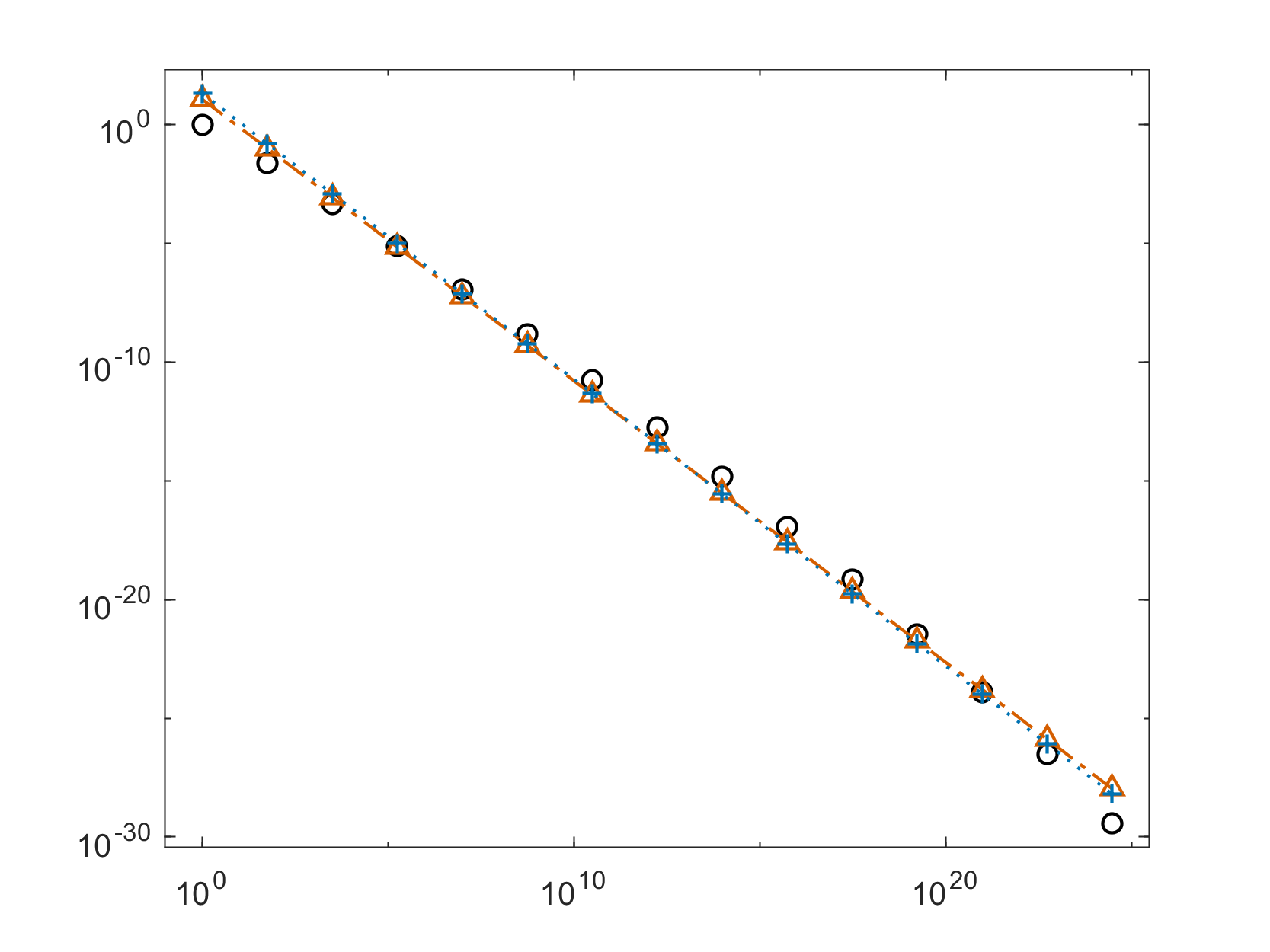}
            \caption{$K = 15,\,m = 55$}
        \end{subfigure}%
        \hfill
        \begin{subfigure}{0.3\textwidth}
            \centering
            \includegraphics[width=\linewidth]{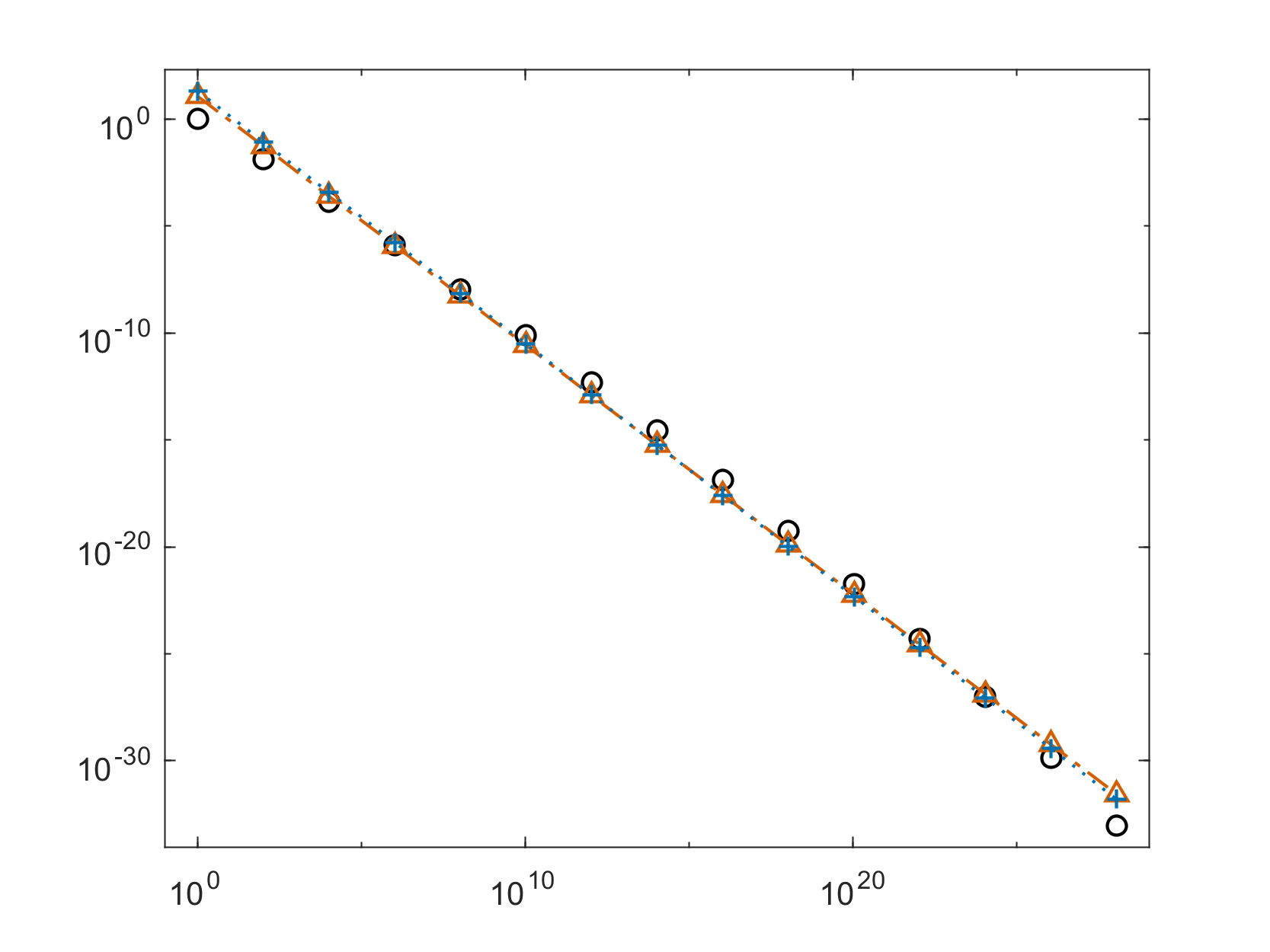}
            \caption{$K = 15,\,m = 100$}
        \end{subfigure}
        
        \begin{subfigure}{0.3\textwidth}
            \centering
            \includegraphics[width=\linewidth]{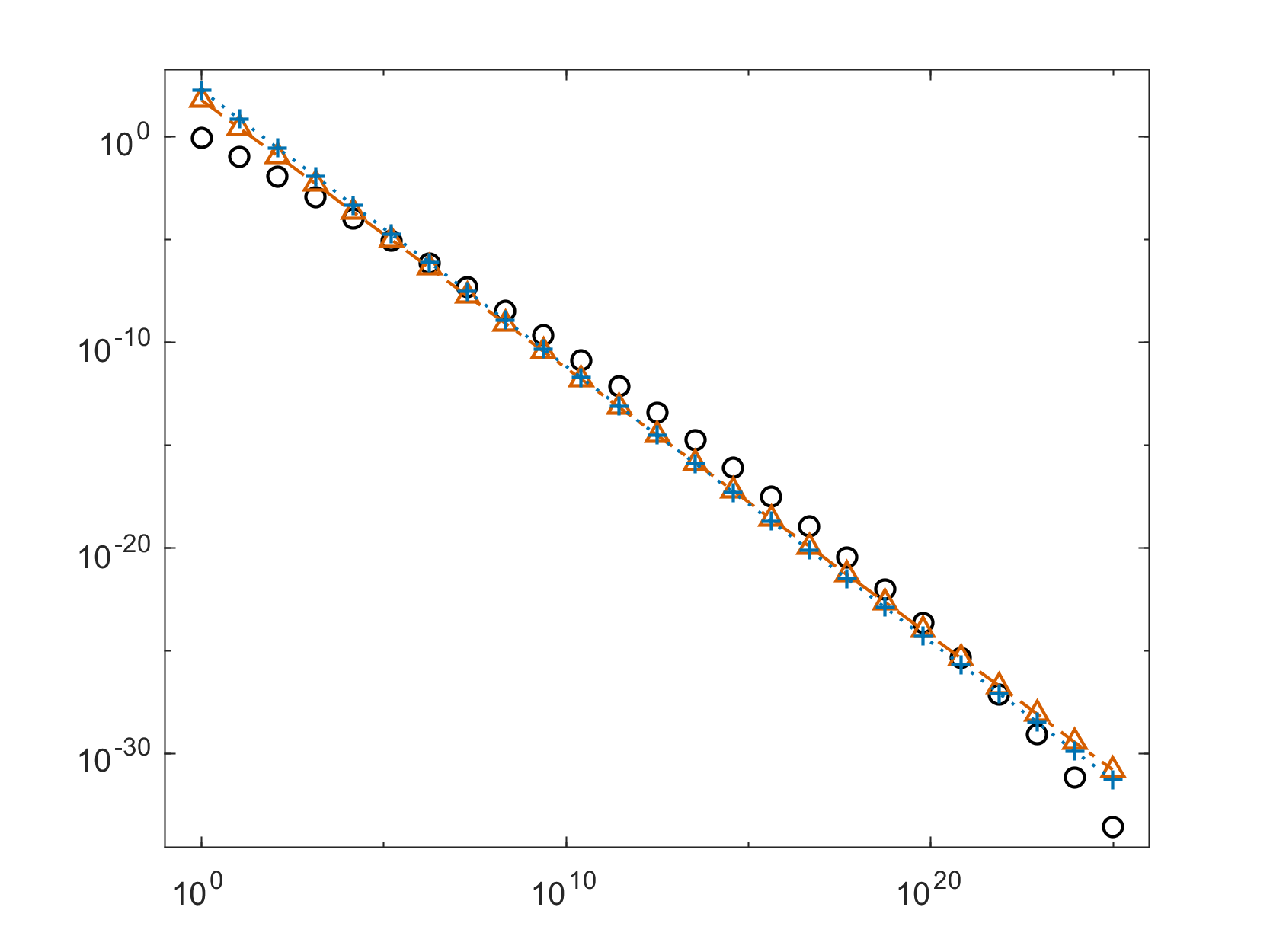}
            \caption{$K = 25,\,m = 10$}
        \end{subfigure}%
        \hfill
        \begin{subfigure}{0.3\textwidth}
            \centering
            \includegraphics[width=\linewidth]{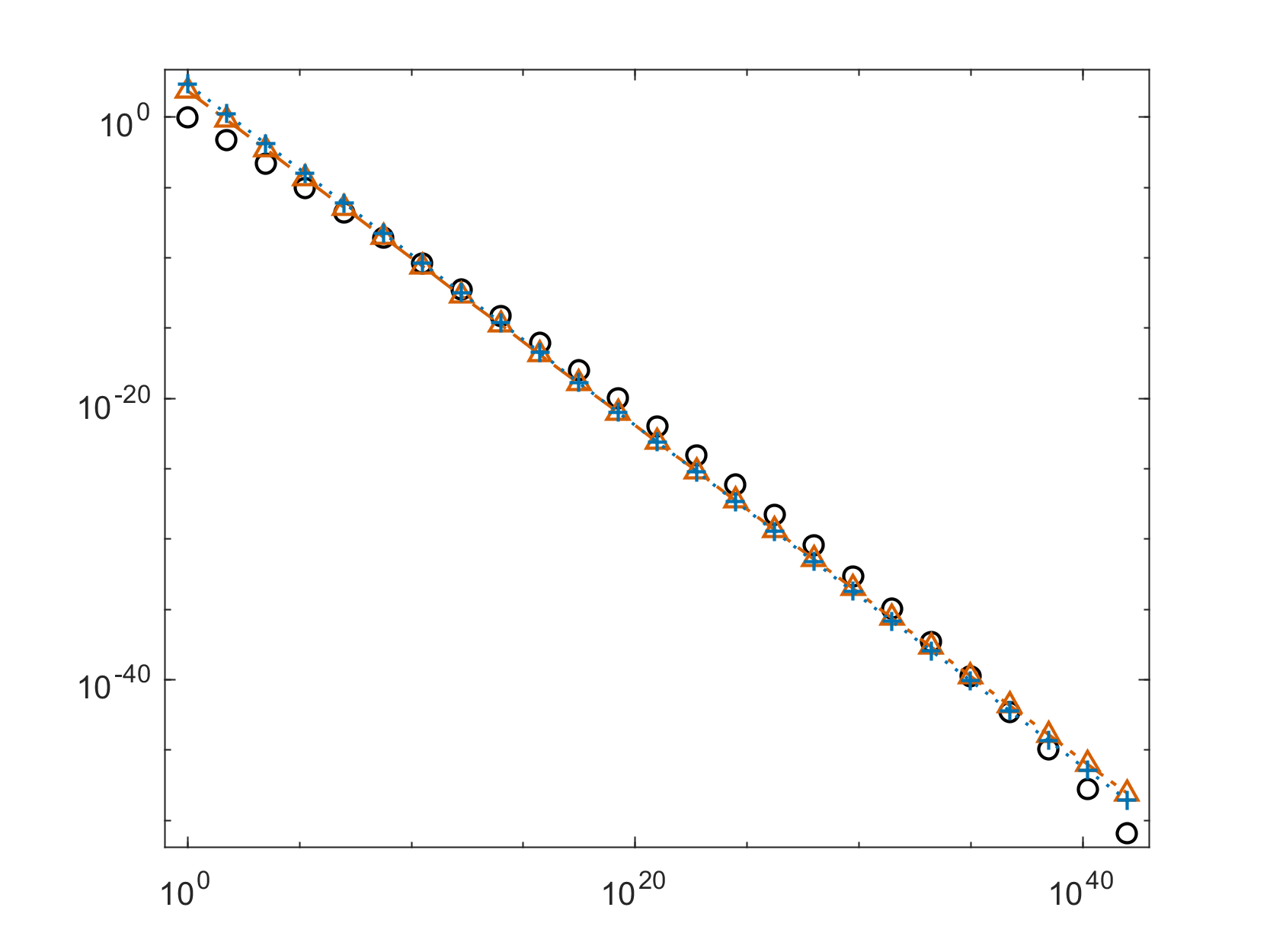}
            \caption{$K = 25,\,m = 55$}
        \end{subfigure}%
        \hfill
        \begin{subfigure}{0.3\textwidth}
            \centering
            \includegraphics[width=\linewidth]{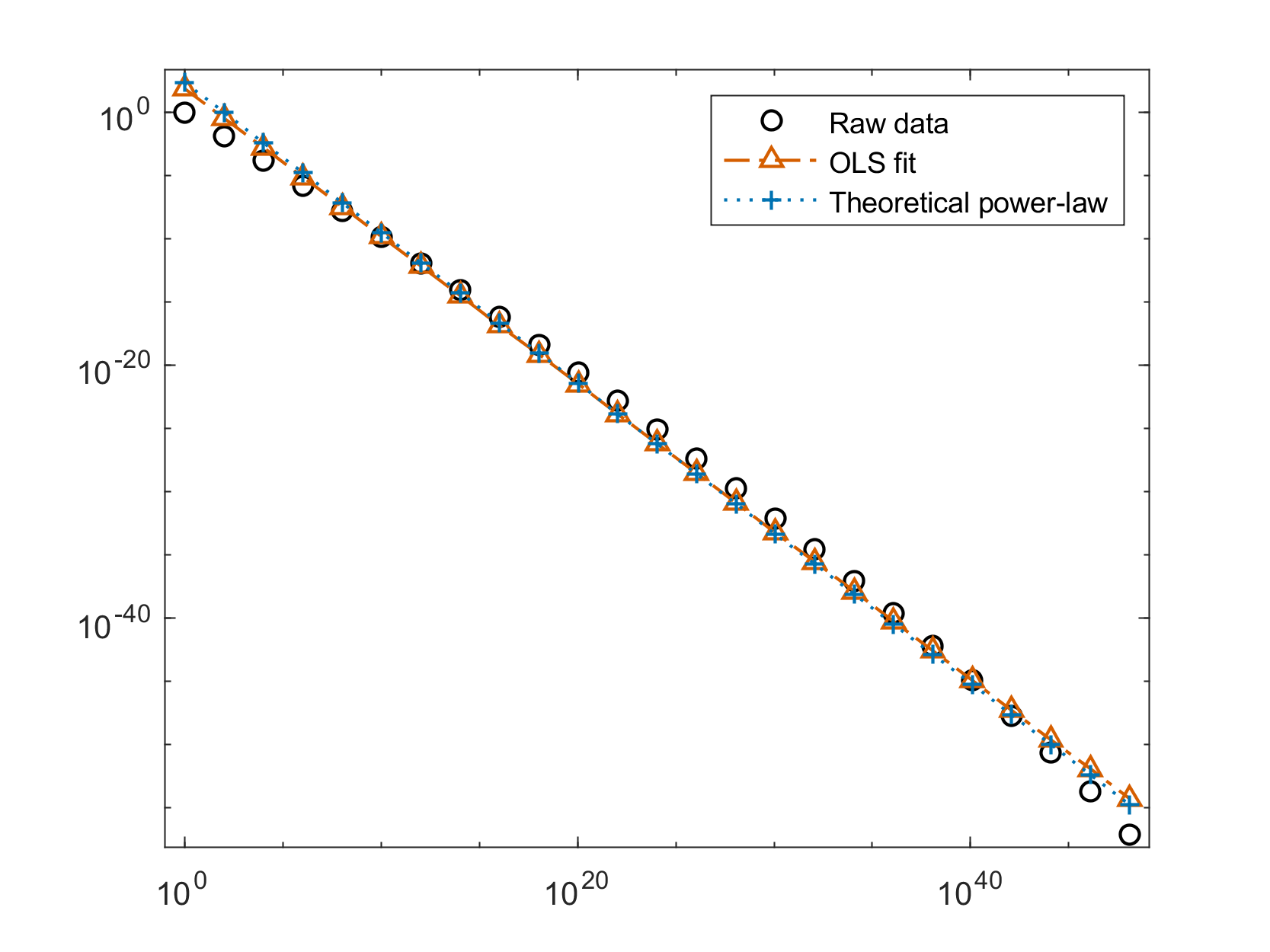}
            \caption{$K = 25,\,m = 100$}
        \end{subfigure}
        
        \vspace{0.5cm}
         Generalized Degrees, $k_{K, 1}$
    \end{minipage}
    
    \caption{\textbf{Plot of the 1-dimensional generalized degree distribution $P_{K, 1}(k)$ for networks under different combinations of parameters $K$ and $m$ in a double-logarithmic coordinate system. }}
    \label{fig:GDD1}
\end{figure}

\begin{figure}[htbp]
    \centering
    
    \begin{minipage}{0.04\textwidth}
        \centering
        \rotatebox{90}{Generalized degree distributions, $P_{K, 2}(k_{K, 2})$}
    \end{minipage}%
    \hfill
    \begin{minipage}{0.94\textwidth}
        \centering
        
        \begin{subfigure}{0.3\textwidth}
            \centering
            \includegraphics[width=\linewidth]{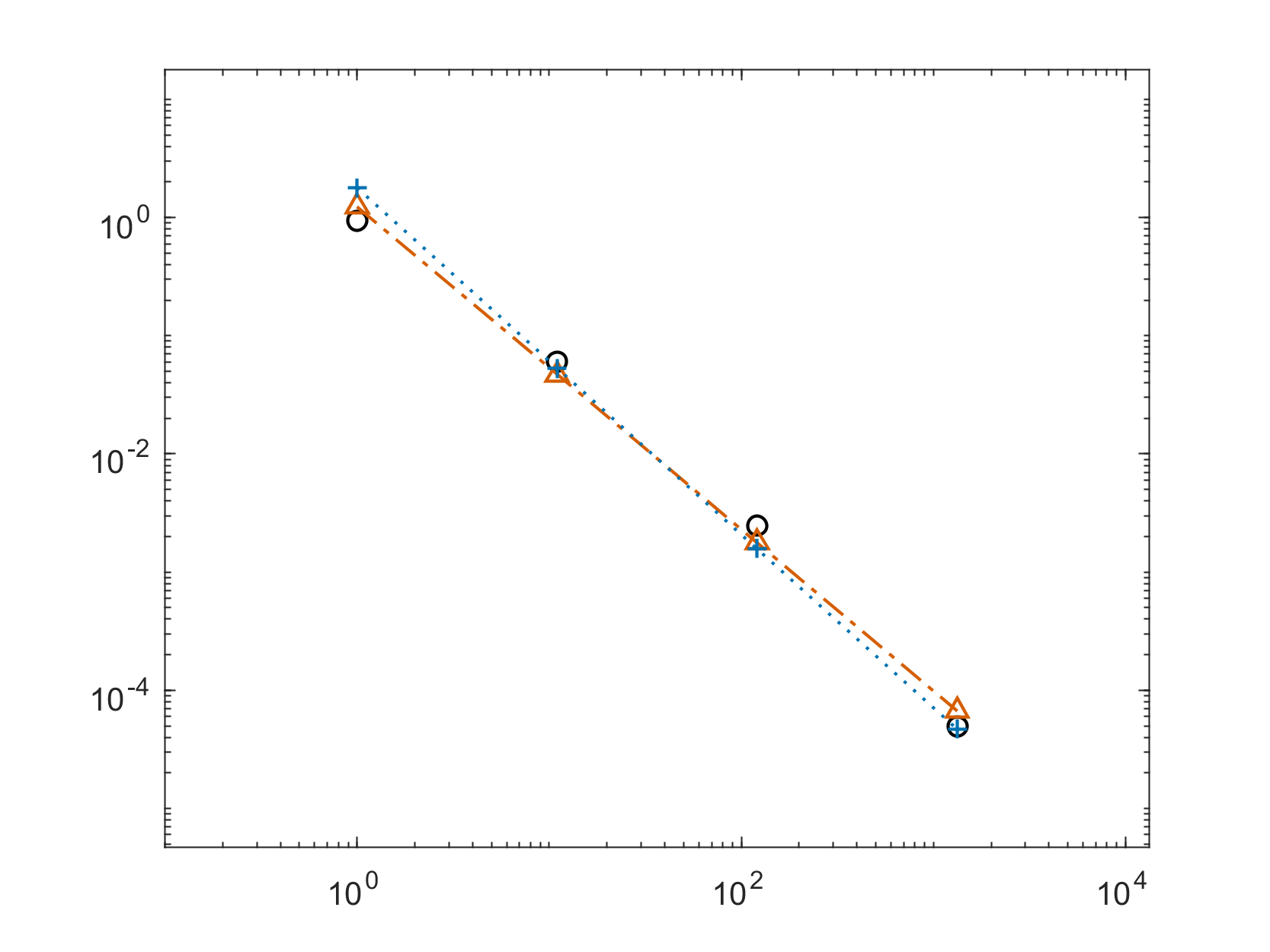}
            \caption{$K = 5,\,m = 10$}
        \end{subfigure}%
        \hfill
        \begin{subfigure}{0.3\textwidth}
            \centering
            \includegraphics[width=\linewidth]{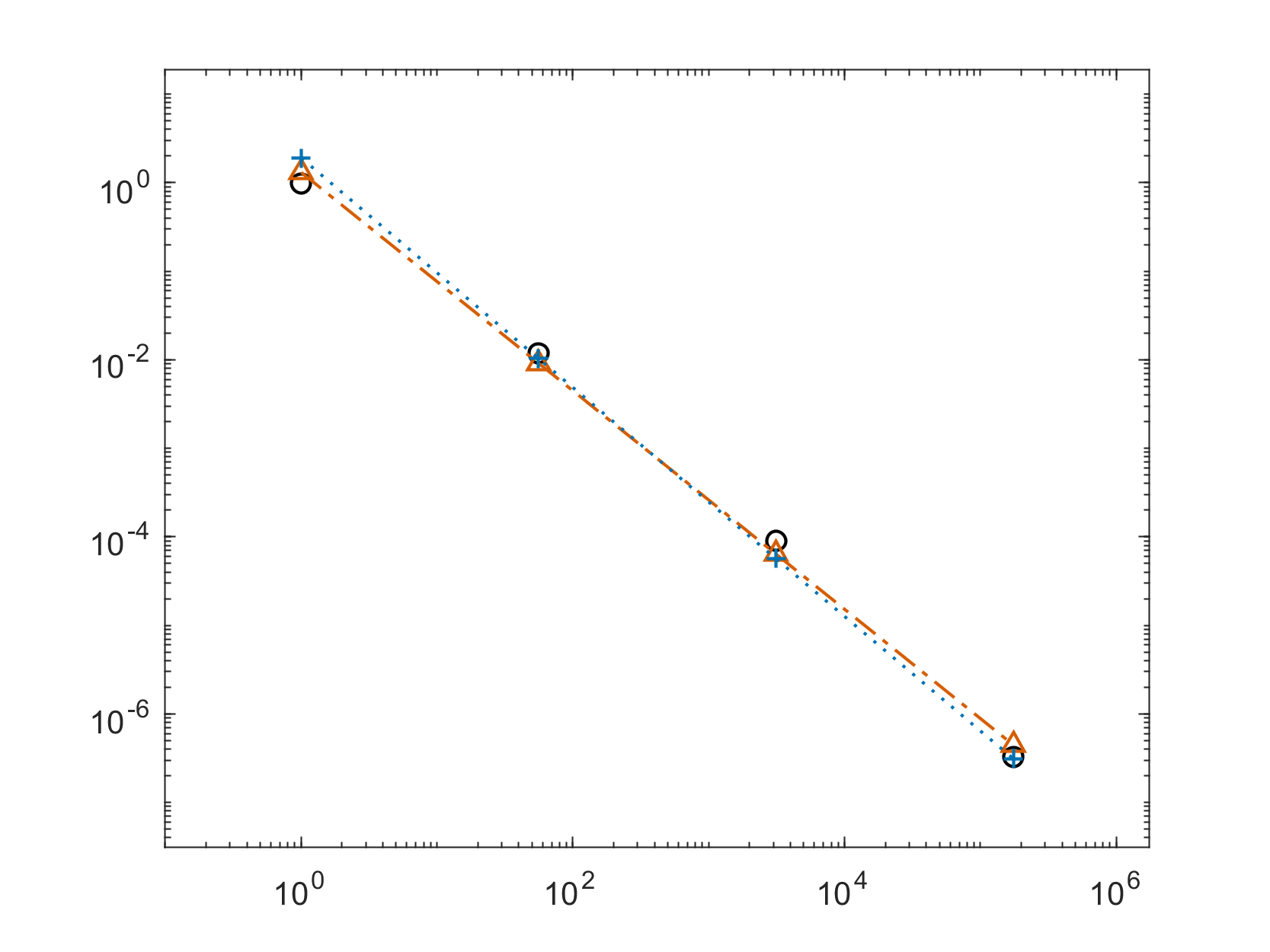}
            \caption{$K = 5,\,m = 55$}
        \end{subfigure}%
        \hfill
        \begin{subfigure}{0.3\textwidth}
            \centering
            \includegraphics[width=\linewidth]{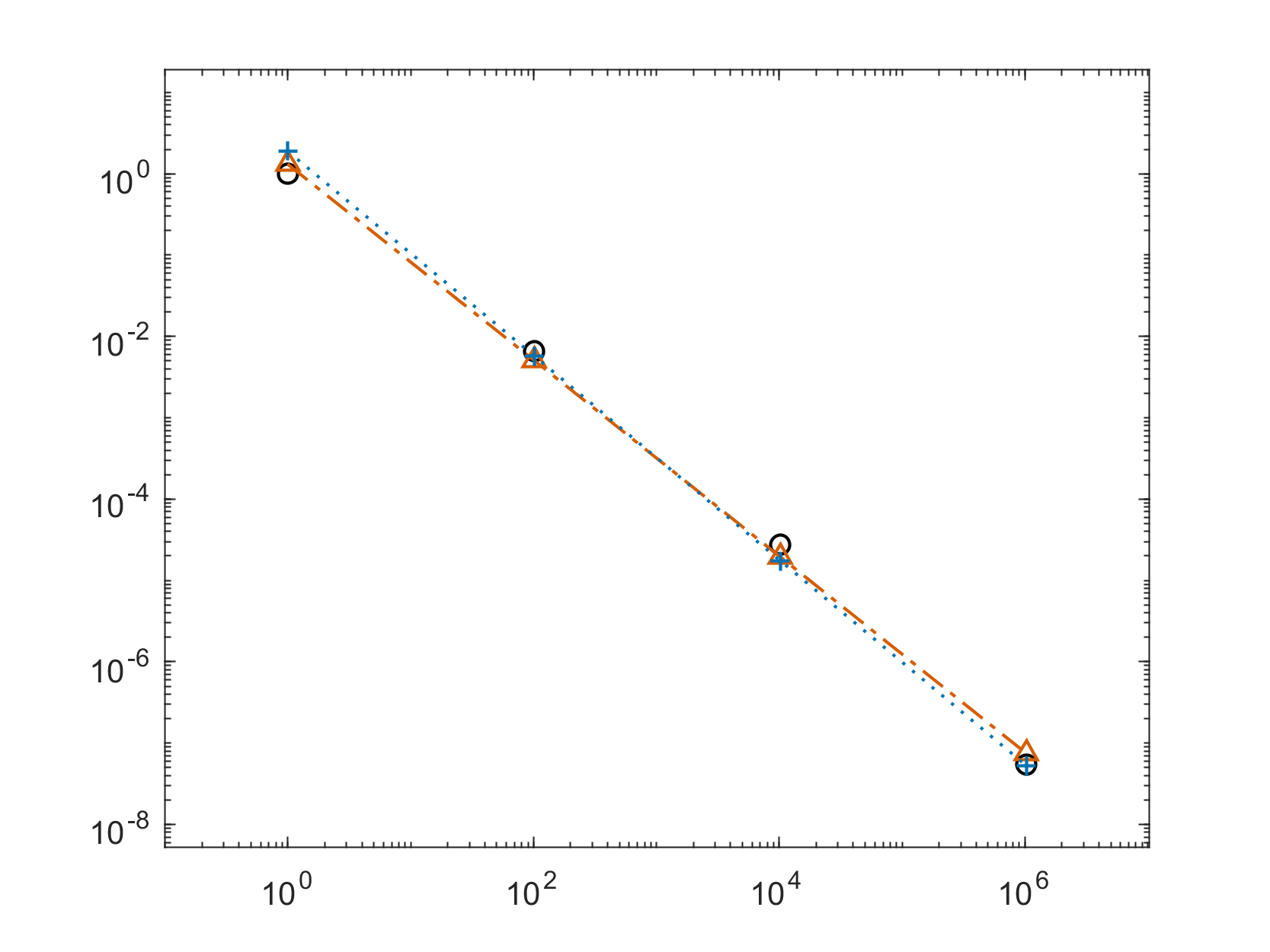}
            \caption{$K = 5,\,m = 100$}
        \end{subfigure}
        
        \begin{subfigure}{0.3\textwidth}
            \centering
            \includegraphics[width=\linewidth]{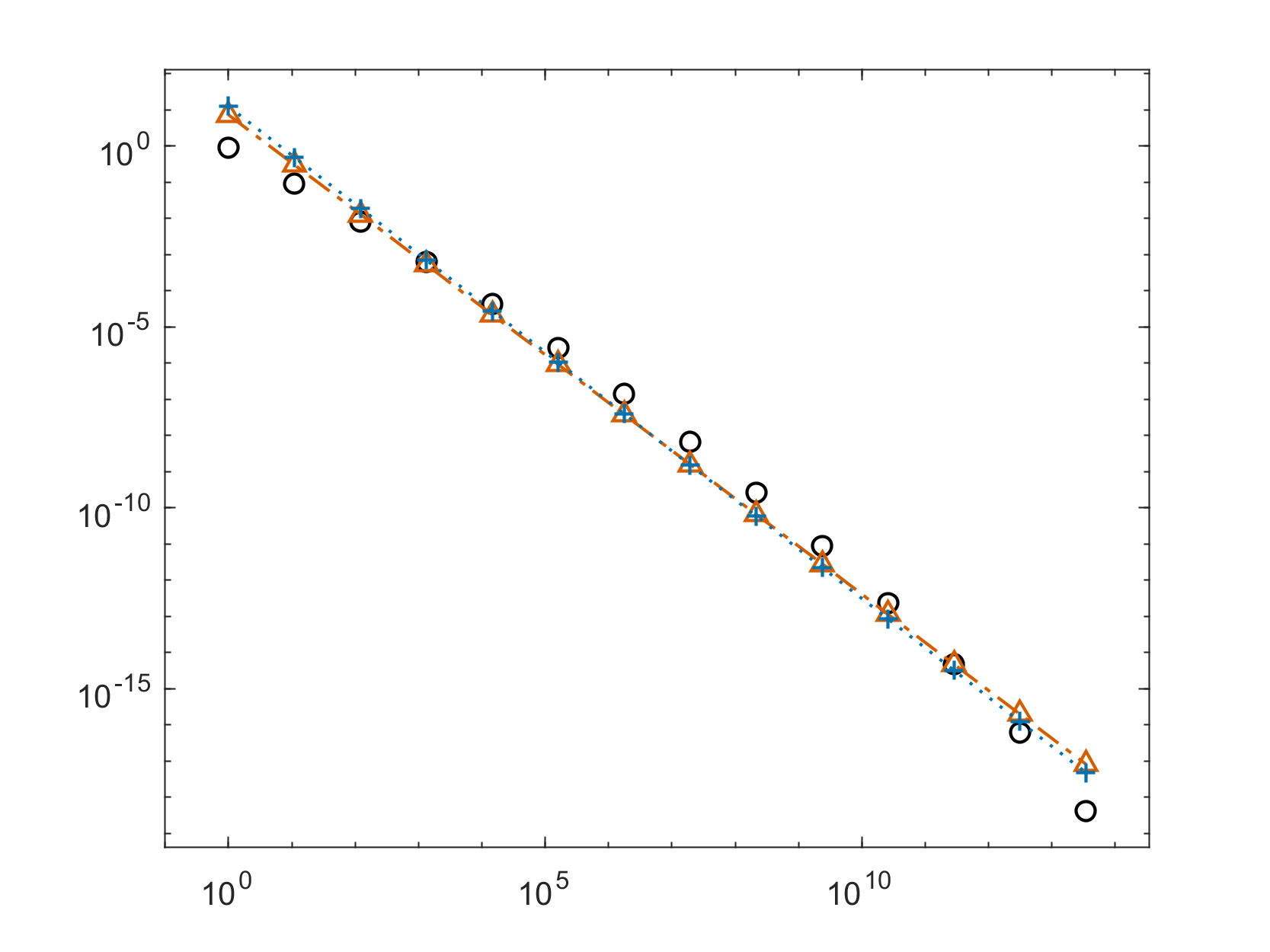}
            \caption{$K = 15,\,m = 10$}
        \end{subfigure}%
        \hfill
        \begin{subfigure}{0.3\textwidth}
            \centering
            \includegraphics[width=\linewidth]{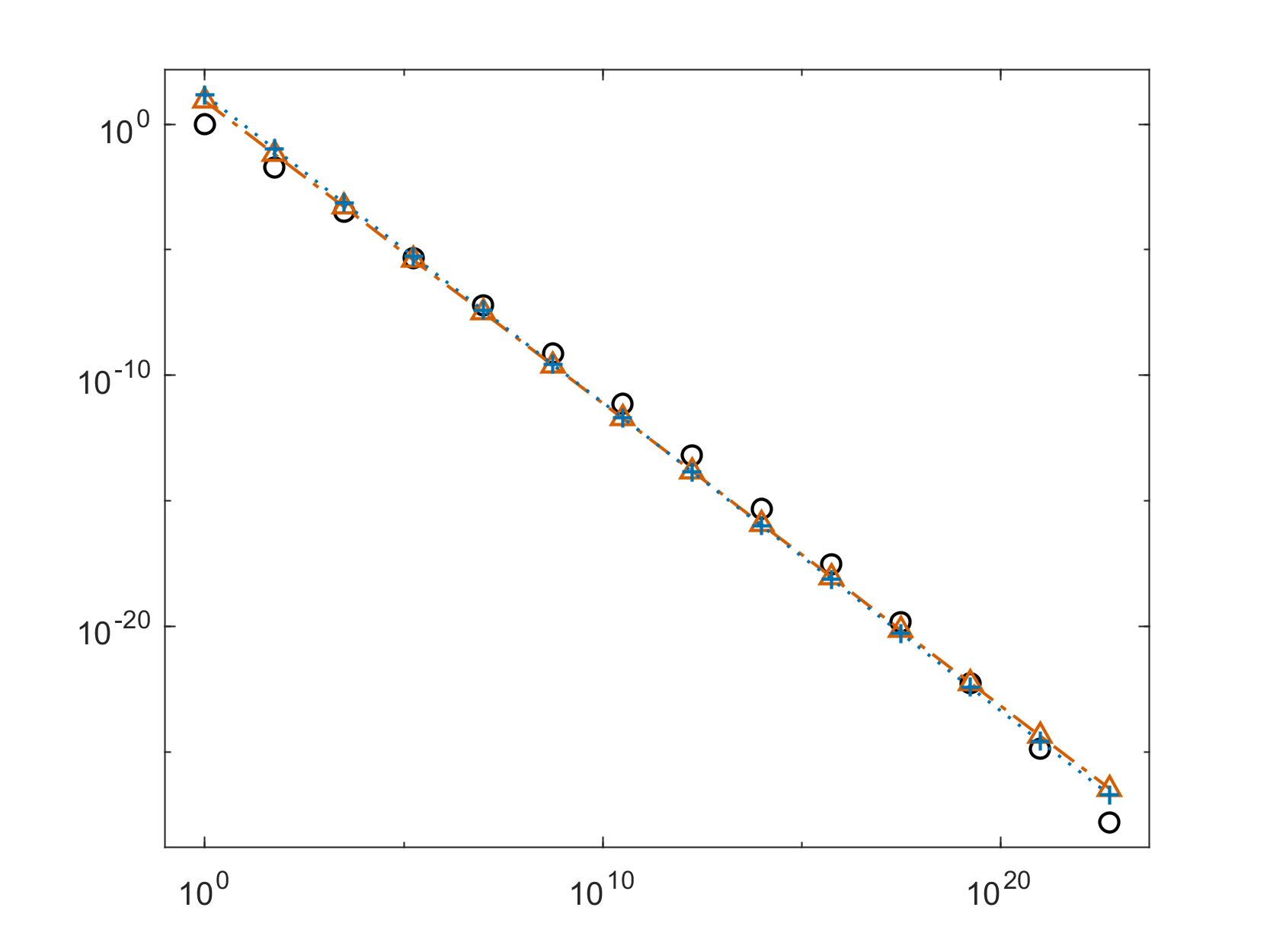}
            \caption{$K = 15,\,m = 55$}
        \end{subfigure}%
        \hfill
        \begin{subfigure}{0.3\textwidth}
            \centering
            \includegraphics[width=\linewidth]{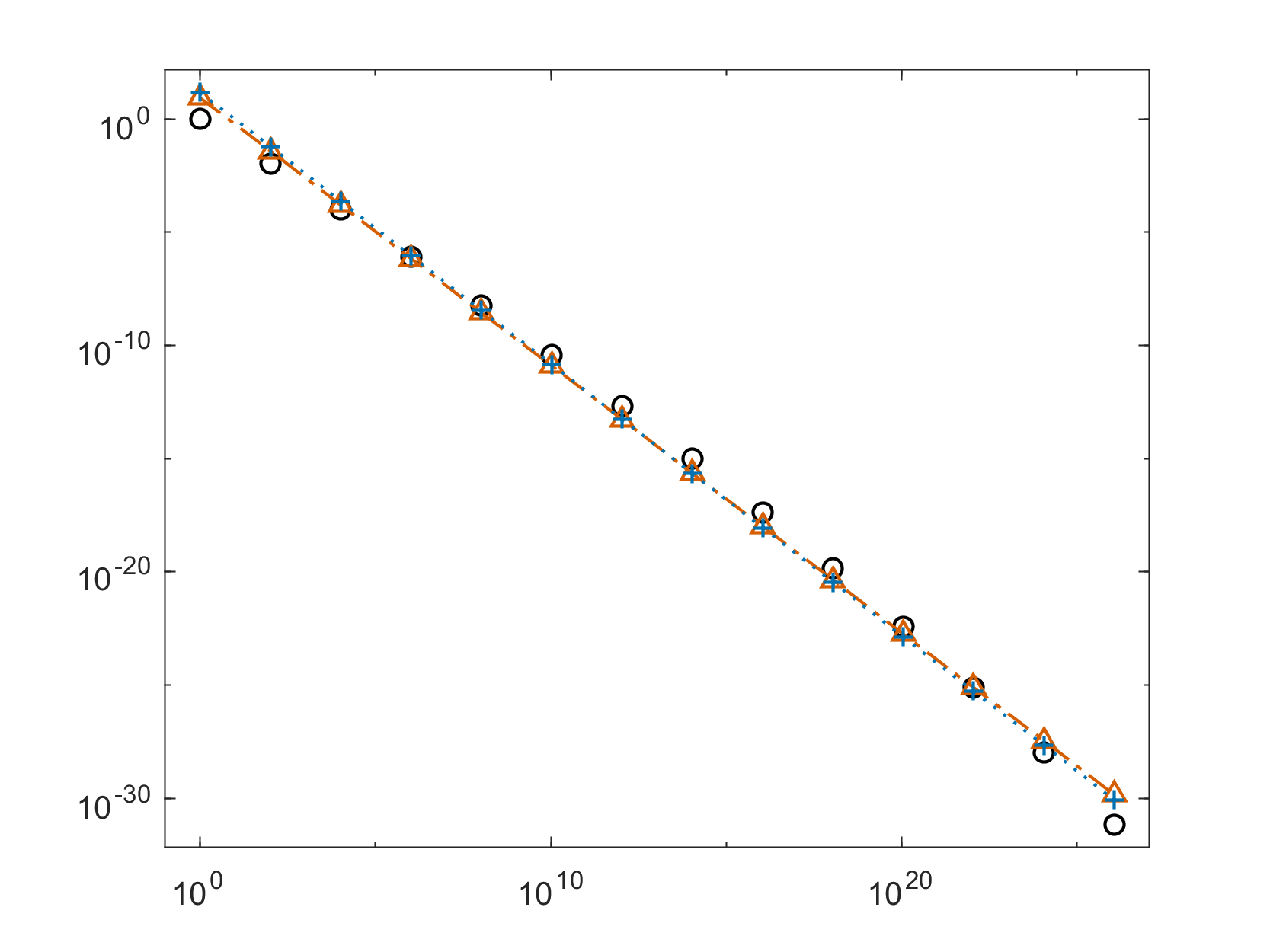}
            \caption{$K = 15,\,m = 100$}
        \end{subfigure}
        
        \begin{subfigure}{0.3\textwidth}
            \centering
            \includegraphics[width=\linewidth]{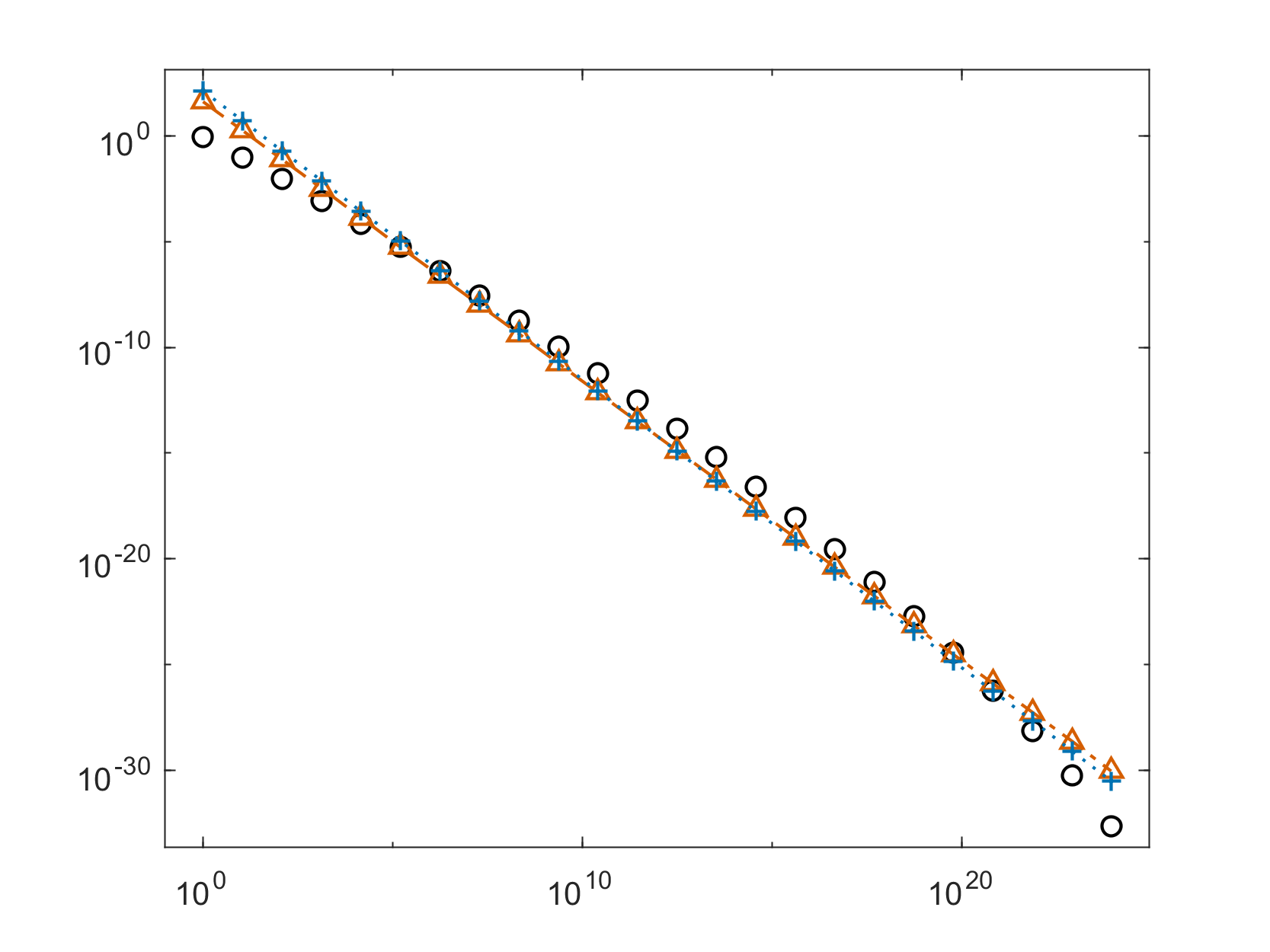}
            \caption{$K = 25,\,m = 10$}
        \end{subfigure}%
        \hfill
        \begin{subfigure}{0.3\textwidth}
            \centering
            \includegraphics[width=\linewidth]{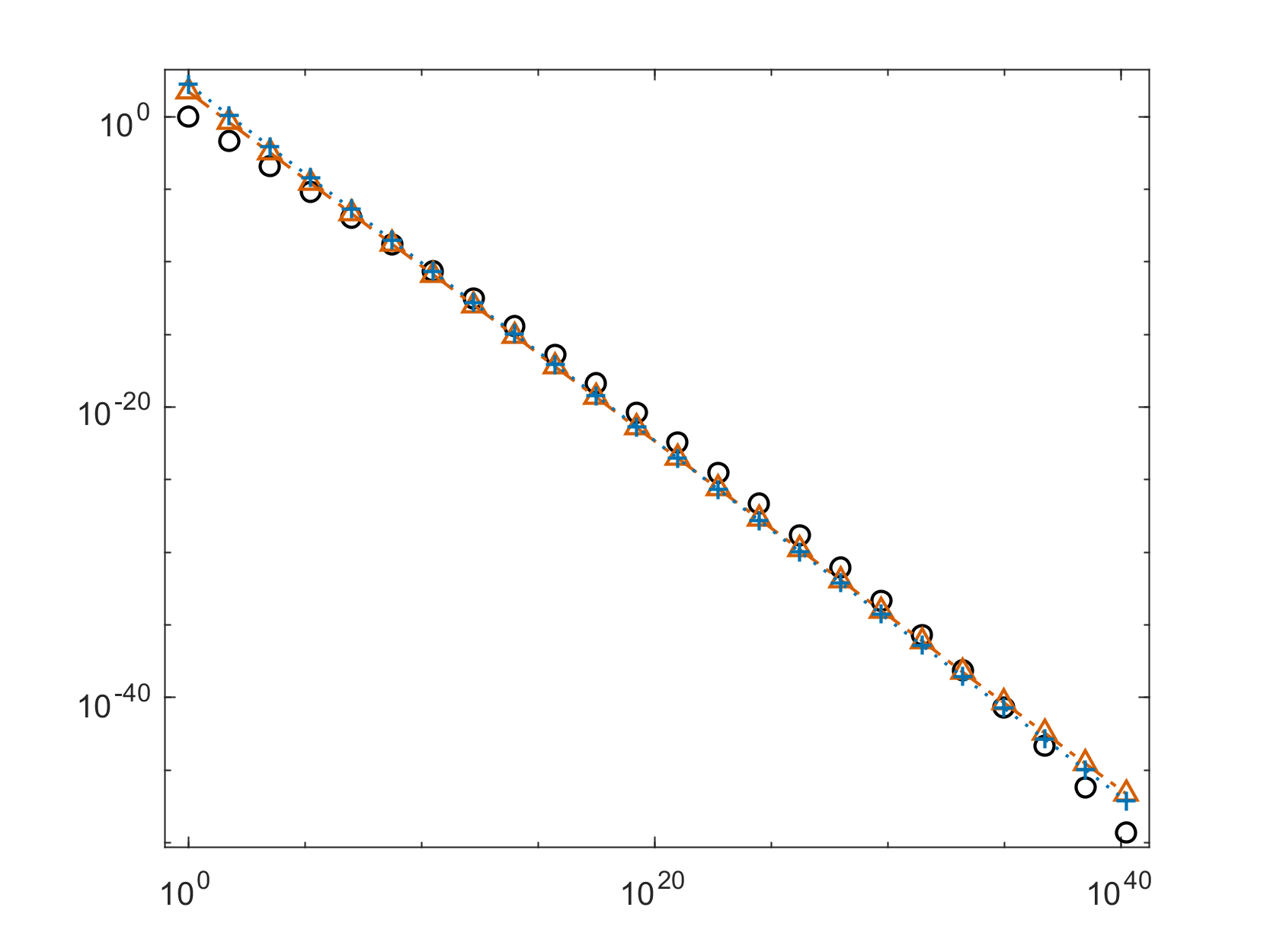}
            \caption{$K = 25,\,m = 55$}
        \end{subfigure}%
        \hfill
        \begin{subfigure}{0.3\textwidth}
            \centering
            \includegraphics[width=\linewidth]{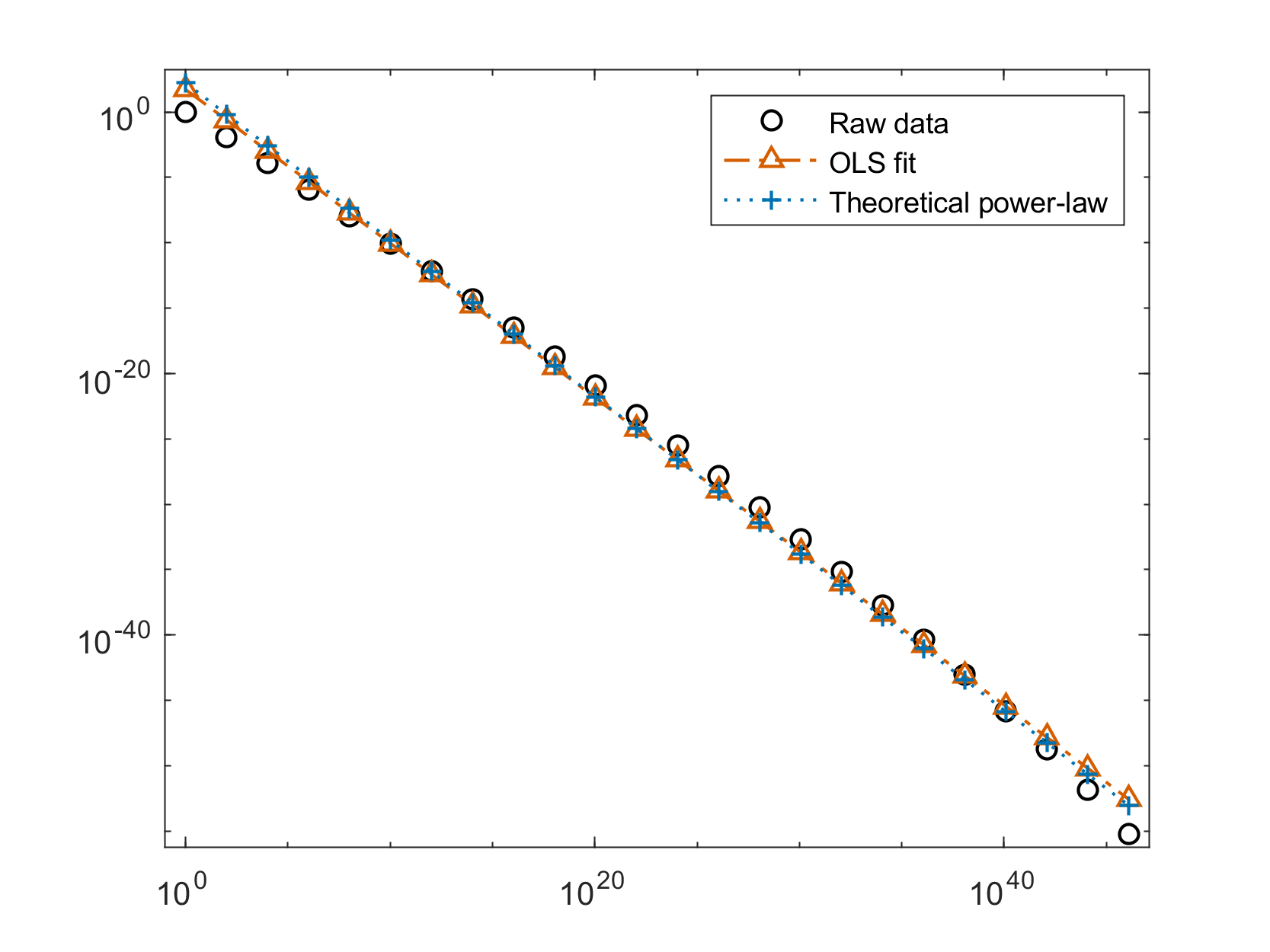}
            \caption{$K = 25,\,m = 100$}
        \end{subfigure}
        
        \vspace{0.5cm}
         Generalized Degrees, $k_{K, 2}$
    \end{minipage}
    
    \caption{\textbf{Plot of the 2-dimensional generalized degree distribution $P_{K, 2}(k)$ for networks under different combinations of parameters $K$ and $m$ in a double-logarithmic coordinate system. }}
    \label{fig:GDD2}
\end{figure}

\section{Conclusion}
This study proposes an iterative generative model for higher-order fractal scale-free networks. First, a fractal network is generated through an iterative process. Subsequently, the clique complex method is employed to derive a higher-order network; more precisely, a fractal pure simplicial complex. The generated network $\mathcal{K}_{t}(K, m)$ is determined by the dimension $K$ of the simplicial complex, the multiplier $m$, and the iteration count $t$. Subsequently, we investigate the properties of the network $\mathcal{K}_{t}(K, m)$. Specifically, we examine:

(1) The number of $K$-dimensional simplices in the network. Based on the network's generation process, the networks produced in this paper consist of $K$-simplices and their faces, thus forming pure simplicial complexes of dimension $K$. Furthermore, the number of $K$-dimensional simplices in the network increases exponentially with the iteration count $t$, amounting to $S^t$;

(2) Fractal characteristics of networks. Theoretical analysis based on similarity dimensions and experimental verification using box-counting dimensions demonstrate that the generated networks exhibit fractal properties, with controllable fractal dimensions; 

(3) Network generalized degree distribution. When the parameter $m$ is large, the networks generated in this paper exhibit scale-free properties, with a generalized degree distribution that is scale-free. The exponent $\gamma$ of the power-law distribution has been determined. 

Future research will continue to explore other features within higher-order fractal networks.

\bibliographystyle{unsrt}
\bibliography{reference.bib}

@article{song_origins_2006,
	title = {Origins of fractality in the growth of complex networks},
	volume = {2},
	copyright = {http://www.springer.com/tdm},
	issn = {1745-2473, 1745-2481},
	url = {https://www.nature.com/articles/nphys266},
	doi = {10.1038/nphys266},
	language = {en},
	number = {4},
	urldate = {2024-05-28},
	journal = {Nature Physics},
	author = {Song, Chaoming and Havlin, Shlomo and Makse, Hernán A.},
	month = apr,
	year = {2006},
	pages = {275--281},
}

@article{qi_internet_2022,
	title = {{AN} {INTERNET} {REVIEW} {TOPIC} {HIERARCHY} {MINING} {METHOD} {BASED} {ON} {MODIFIED} {CONTINUOUS} {RENORMALIZATION} {PROCEDURE}},
	volume = {30},
	issn = {0218-348X, 1793-6543},
	url = {https://www.worldscientific.com/doi/10.1142/S0218348X22501341},
	doi = {10.1142/S0218348X22501341},
	language = {en},
	number = {07},
	urldate = {2025-11-02},
	journal = {Fractals},
	author = {Qi, Lin and Guo, Fei-Yan and Zhang, Jian and Wang, Yu-Wei},
	month = nov,
	year = {2022},
	pages = {2250134},
}

@article{ma_understanding_2024,
	title = {Understanding influence of fractal generative manner on structural properties of tree networks},
	volume = {180},
	issn = {09600779},
	url = {https://linkinghub.elsevier.com/retrieve/pii/S0960077924000742},
	doi = {10.1016/j.chaos.2024.114523},
	language = {en},
	urldate = {2025-10-15},
	journal = {Chaos, Solitons \& Fractals},
	author = {Ma, Fei and Wang, Ping},
	month = mar,
	year = {2024},
	pages = {114523},
}

@article{lambiotte_networks_2019,
	title = {From networks to optimal higher-order models of complex systems},
	volume = {15},
	issn = {1745-2473, 1745-2481},
	url = {https://www.nature.com/articles/s41567-019-0459-y},
	doi = {10.1038/s41567-019-0459-y},
	language = {en},
	number = {4},
	urldate = {2025-09-30},
	journal = {Nature Physics},
	author = {Lambiotte, Renaud and Rosvall, Martin and Scholtes, Ingo},
	month = apr,
	year = {2019},
	pages = {313--320},
}

@article{benson_higher-order_2016,
	title = {Higher-order organization of complex networks},
	volume = {353},
	issn = {0036-8075, 1095-9203},
	url = {https://www.science.org/doi/10.1126/science.aad9029},
	doi = {10.1126/science.aad9029},
	language = {en},
	number = {6295},
	urldate = {2025-10-09},
	journal = {Science},
	author = {Benson, Austin R. and Gleich, David F. and Leskovec, Jure},
	month = jul,
	year = {2016},
	pages = {163--166},
}

@article{sizemore_knowledge_2018,
	title = {Knowledge gaps in the early growth of semantic feature networks},
	volume = {2},
	issn = {2397-3374},
	url = {https://doi.org/10.1038/s41562-018-0422-4},
	doi = {10.1038/s41562-018-0422-4},
	number = {9},
	journal = {Nature Human Behaviour},
	author = {Sizemore, Ann E. and Karuza, Elisabeth A. and Giusti, Chad and Bassett, Danielle S.},
	month = sep,
	year = {2018},
	pages = {682--692},
}

@article{giusti_twos_2016,
	title = {Two’s company, three (or more) is a simplex: {Algebraic}-topological tools for understanding higher-order structure in neural data},
	volume = {41},
	issn = {0929-5313, 1573-6873},
	shorttitle = {Two’s company, three (or more) is a simplex},
	url = {http://link.springer.com/10.1007/s10827-016-0608-6},
	doi = {10.1007/s10827-016-0608-6},
	language = {en},
	number = {1},
	urldate = {2025-10-20},
	journal = {Journal of Computational Neuroscience},
	author = {Giusti, Chad and Ghrist, Robert and Bassett, Danielle S.},
	month = aug,
	year = {2016},
	pages = {1--14},
}

@article{tadic_multiscale_2022,
	title = {Multiscale fractality in partial phase synchronisation on simplicial complexes around brain hubs},
	volume = {160},
	issn = {09600779},
	url = {https://linkinghub.elsevier.com/retrieve/pii/S0960077922004118},
	doi = {10.1016/j.chaos.2022.112201},
	language = {en},
	urldate = {2025-10-29},
	journal = {Chaos, Solitons \& Fractals},
	author = {Tadić, Bosiljka and Chutani, Malayaja and Gupte, Neelima},
	month = jul,
	year = {2022},
	pages = {112201},
}

@article{luo_fractal_2024,
	title = {Fractal information dissemination and clustering evolution on social hypernetwork},
	volume = {34},
	issn = {1054-1500, 1089-7682},
	url = {https://pubs.aip.org/cha/article/34/9/093128/3313313/Fractal-information-dissemination-and-clustering},
	doi = {10.1063/5.0228903},
	language = {en},
	number = {9},
	urldate = {2025-10-29},
	journal = {Chaos: An Interdisciplinary Journal of Nonlinear Science},
	author = {Luo, Li and Nian, Fuzhong and Cui, Yuanlin and Li, Fangfang},
	month = sep,
	year = {2024},
	pages = {093128},
}

@INPROCEEDINGS{8100392,
  author={Potebnia, Artem},
  booktitle={2017 IEEE First Ukraine Conference on Electrical and Computer Engineering (UKRCON)}, 
  title={Formation of the multifractal hypergraph structure reflecting the self-similarity properties of the computational complexity classes}, 
  year={2017},
  volume={},
  number={},
  pages={953-958},
  doi={10.1109/UKRCON.2017.8100392}}

@article{ji_astute_2022,
	title = {Astute {Video} {Transmission} for {Geographically} {Dispersed} {Devices} in {Visual} {IoT} {Systems}},
	volume = {21},
	copyright = {https://ieeexplore.ieee.org/Xplorehelp/downloads/license-information/IEEE.html},
	issn = {1536-1233, 1558-0660, 2161-9875},
	url = {https://ieeexplore.ieee.org/document/9142435/},
	doi = {10.1109/TMC.2020.3009745},
	number = {2},
	urldate = {2025-10-29},
	journal = {IEEE Transactions on Mobile Computing},
	author = {Ji, Wen and Duan, Lingyu and Huang, Xi and Chai, Yueting},
	month = feb,
	year = {2022},
	pages = {448--464},
}

@article{ma_fractal_2021,
	title = {A fractal hypernetwork model with good controllability},
	volume = {6},
	issn = {2473-6988},
	url = {http://www.aimspress.com/article/doi/10.3934/math.2021799},
	doi = {10.3934/math.2021799},
	language = {en},
	number = {12},
	urldate = {2025-10-29},
	journal = {AIMS Mathematics},
	author = {Ma, Xiujuan and Ma, Fuxiang and Yin, Jun},
	year = {2021},
	pages = {13758--13773},
}

@article{zheng_observation_2022,
	title = {Observation of fractal higher-order topological states in acoustic metamaterials},
	volume = {67},
	issn = {20959273},
	url = {https://linkinghub.elsevier.com/retrieve/pii/S2095927322004273},
	doi = {10.1016/j.scib.2022.09.020},
	language = {en},
	number = {20},
	urldate = {2025-10-29},
	journal = {Science Bulletin},
	author = {Zheng, Shengjie and Man, Xianfeng and Kong, Ze-Lin and Lin, Zhi-Kang and Duan, Guiju and Chen, Ning and Yu, Dejie and Jiang, Jian-Hua and Xia, Baizhan},
	month = oct,
	year = {2022},
	pages = {2069--2075},
}

@article{xie_combinatorial_2023,
	title = {Combinatorial {Properties} for a {Class} of {Simplicial} {Complexes} {Extended} from {Pseudo}-fractal {Scale}-free {Web}},
	volume = {31},
	issn = {0218-348X, 1793-6543},
	url = {http://arxiv.org/abs/2301.03230},
	doi = {10.1142/S0218348X23500226},
	language = {en},
	number = {03},
	urldate = {2025-10-29},
	journal = {Fractals},
	author = {Xie, Zixuan and Wang, Yucheng and Xu, Wanyue and Zhu, Liwang and Li, Wei and Zhang, Zhongzhi},
	month = jan,
	year = {2023},
	pages = {2350022},
}

@book{bianconi_higher-order_2021,
	edition = {1},
	title = {Higher-{Order} {Networks}},
	copyright = {https://www.cambridge.org/core/terms},
	isbn = {978-1-108-77099-6 978-1-108-72673-3},
	url = {https://www.cambridge.org/core/product/identifier/9781108770996/type/element},
	language = {en},
	urldate = {2025-10-16},
	publisher = {Cambridge University Press},
	author = {Bianconi, Ginestra},
	month = dec,
	year = {2021},
	doi = {10.1017/9781108770996},
}

@article{bianconi_network_2016,
  title = {Network geometry with flavor: From complexity to quantum geometry},
  author = {Bianconi, Ginestra and Rahmede, Christoph},
  journal = {Physical Review E},
  volume = {93},
  issue = {3},
  pages = {032315},
  numpages = {15},
  year = {2016},
  month = mar,
  publisher = {American Physical Society},
  doi = {10.1103/PhysRevE.93.032315},
  url = {https://link.aps.org/doi/10.1103/PhysRevE.93.032315}
}

@article{gallos_ReviewFractality_2007,
  title = {A Review of Fractality and Self-Similarity in Complex Networks},
  author = {Gallos, Lazaros K. and Song, Chaoming and Makse, Hern{\'a}n A.},
  year = 2007,
  month = dec,
  journal = {Physica A: Statistical Mechanics and its Applications},
  volume = {386},
  number = {2},
  pages = {686--691},
  issn = {03784371},
  doi = {10.1016/j.physa.2007.07.069},
  urldate = {2025-11-04},
  copyright = {https://www.elsevier.com/tdm/userlicense/1.0/}
}

@article{song_how_2007,
	title = {How to calculate the fractal dimension of a complex network: the box covering algorithm},
	volume = {2007},
	issn = {1742-5468},
	shorttitle = {How to calculate the fractal dimension of a complex network},
	url = {https://iopscience.iop.org/article/10.1088/1742-5468/2007/03/P03006},
	doi = {10.1088/1742-5468/2007/03/P03006},
	language = {en},
	number = {03},
	urldate = {2025-11-02},
	journal = {Journal of Statistical Mechanics: Theory and Experiment},
	author = {Song, Chaoming and Gallos, Lazaros K and Havlin, Shlomo and Makse, Hernán A},
	month = mar,
	year = {2007},
	pages = {P03006--P03006},
}

@article{sun_overlapping-box-covering_2014,
	title = {Overlapping-box-covering method for the fractal dimension of complex networks},
	volume = {89},
	copyright = {http://link.aps.org/licenses/aps-default-license},
	issn = {1539-3755, 1550-2376},
	url = {https://link.aps.org/doi/10.1103/PhysRevE.89.042809},
	doi = {10.1103/PhysRevE.89.042809},
	language = {en},
	number = {4},
	urldate = {2024-05-28},
	journal = {Physical Review E},
	author = {Sun, Yuanyuan and Zhao, Yujie},
	month = apr,
	year = {2014},
	pages = {042809},
}

@article{guo_hub-collision_2022,
	title = {Hub-collision avoidance and leaf-node options algorithm for fractal dimension and renormalization of complex networks},
	volume = {32},
	issn = {1054-1500, 1089-7682},
	url = {https://pubs.aip.org/cha/article/32/12/123116/2836000/Hub-collision-avoidance-and-leaf-node-options},
	doi = {10.1063/5.0113001},
	language = {en},
	number = {12},
	urldate = {2024-12-02},
	journal = {Chaos: An Interdisciplinary Journal of Nonlinear Science},
	author = {Guo, Fei-Yan and Zhou, Jia-Jun and Ruan, Zhong-Yuan and Zhang, Jian and Qi, Lin},
	month = dec,
	year = {2022},
	pages = {123116},
}

@article{wen_fractal_2021,
	title = {The fractal dimension of complex networks: {A} review},
	volume = {73},
	issn = {15662535},
	shorttitle = {The fractal dimension of complex networks},
	url = {https://linkinghub.elsevier.com/retrieve/pii/S1566253521000166},
	doi = {10.1016/j.inffus.2021.02.001},
	language = {en},
	urldate = {2024-05-28},
	journal = {Information Fusion},
	author = {Wen, Tao and Cheong, Kang Hao},
	month = sep,
	year = {2021},
	pages = {87--102},
}

@article{zakar-polyak_towards_2023,
	title = {Towards a better understanding of the characteristics of fractal networks},
	volume = {8},
	issn = {2364-8228},
	url = {https://appliednetsci.springeropen.com/articles/10.1007/s41109-023-00537-8},
	doi = {10.1007/s41109-023-00537-8},
	language = {en},
	number = {1},
	urldate = {2024-11-11},
	journal = {Applied Network Science},
	author = {Zakar-Polyák, Enikő and Nagy, Marcell and Molontay, Roland},
	month = mar,
	year = {2023},
	pages = {17},
}

@book{10.1093/oso/9780198805090.001.0001,
    author = {Newman, Mark},
    title = {Networks},
    publisher = {Oxford University Press},
    year = {2018},
    month = {07},
    isbn = {9780198805090},
    doi = {10.1093/oso/9780198805090.001.0001},
    url = {https://doi.org/10.1093/oso/9780198805090.001.0001},
}

@article{doi:10.1126/science.286.5439.509,
author = {Albert-László Barabási  and Réka Albert },
title = {Emergence of Scaling in Random Networks},
journal = {Science},
volume = {286},
number = {5439},
pages = {509-512},
year = {1999},
doi = {10.1126/science.286.5439.509},
URL = {https://www.science.org/doi/abs/10.1126/science.286.5439.509},
eprint = {https://www.science.org/doi/pdf/10.1126/science.286.5439.509},
}

@article{albert_diameter_1999,
	title = {Diameter of the {World}-{Wide} {Web}},
	volume = {401},
	issn = {1476-4687},
	url = {https://doi.org/10.1038/43601},
	doi = {10.1038/43601},
	number = {6749},
	journal = {Nature},
	author = {Albert, Réka and Jeong, Hawoong and Barabási, Albert-László},
	month = sep,
	year = {1999},
	pages = {130--131},
}

\end{document}